\documentclass{aimshi}
\usepackage{amsmath}
\usepackage{paralist}
\usepackage{graphicx}\usepackage{epstopdf} 
\usepackage[colorlinks=true]{hyperref}
\hypersetup{urlcolor=blue, citecolor=red}

\def\currentvolume{X}

\newtheorem{theorem}{Theorem}[section]

\newtheorem{lemma}[theorem]{Lemma}

\theoremstyle{definition}

\newtheorem{remark}{Remark}

\newcommand{\Rset}{\mathbb{R}}
\newcommand{\calK}{\mathcal{K}}
\newcommand{\calN}{\mathcal{N}}
\newcommand{\calP}{\mathcal{P}}
\newcommand{\calKL}{\mathcal{KL}}
\newcommand{\bfA}{\mathbf{A}}
\newcommand{\bfB}{\mathbf{B}}
\newcommand{\bfC}{\mathbf{C}}
\newcommand{\bfD}{\mathbf{D}}
\newcommand{\bfE}{\mathbf{E}}
\newcommand{\bfF}{\mathbf{F}}

\DeclareMathOperator*{\esssup}{ess\,sup}

\title[Input-to-State Stability and Lyapunov Functions for SIR]%
{I\lowercase{nput-to-State Stability and \uppercase{L}yapunov Functions with Explicit Domains for \uppercase{SIR} Model of Infectious Diseases}}

\author[Hirosih Ito]{}


\email{hiroshi@ces.kyutech.ac.jp}

\thanks{The author is supported by JSPS KAKENHI Grant Number 20K04536.}

\begin{document}
\maketitle

\centerline{\scshape Hiroshi Ito$^*$}
\medskip
{\footnotesize
 \centerline{Department of Intelligent and Control Systems}
   \centerline{Kyushu Institute of Technology}
   \centerline{680-4 Kawazu, Iizuka 820-8502, Japan}
} 

\medskip


\bigskip


\begin{abstract}
This paper demonstrates input-to-state  stability (ISS) of the SIR model of 
infectious diseases with respect to the disease-free equilibrium and 
the endemic equilibrium. Lyapunov functions are constructed to verify 
that both equilibria are individually robust with respect to perturbation of 
newborn/immigration rate which determines the eventual state of 
populations in epidemics. The construction and analysis are geometric and global
in the space of the populations. In addition to the establishment of ISS, this paper 
shows how explicitly the constructed level sets reflect the flow of trajectories. 
Essential obstacles and keys for the construction of Lyapunov functions 
are elucidated. 
The proposed Lyapunov functions which have strictly negative 
derivative allow us to not only establish ISS, but also 
get rid of the use of LaSalle's invariance principle and 
popular simplifying assumptions. 

\end{abstract}

\section{Introduction}

For infectious diseases, mathematical models play two major roles in helping 
epidemiologist and societies design schemes aiming to improve control or 
eradicate the infection from population \cite{Keelinfdiseasbook09}. 
One role is quantitative prediction in which its accuracy is the primary concern. 
The other is qualitative understanding of epidemiological processes. 
For the latter, analytical studies on simple models have been providing 
generic interpretations of behavior of diseases transmission and spread. 
This paper pursues this direction by focusing on the popular model 
called the SIR model \cite{Dietz67,KermackEpidemics27}. 

The SIR model has an endemic equilibrium and a disease-free equilibrium. 
If the newborn rate is large in the population, the endemic equilibrium emerges and 
the trajectory of populational behavior heads for the equilibrium. 
Here, the newborn rate is the external signal flowing into the SIR model, 
and it describes not only birth, but also the susceptible flux entering 
the area to which populations of interest belongs, i.e., immigration of 
susceptible individuals. 

Stability is a fundamental concept that characterizes behavior 
of dynamics for each equilibrium. Roughly, asymptotic stability gives 
a guarantee that trajectories starting sufficiently near the target 
equilibrium converges to the equilibrium. 
Jacobian linearization, which is called 
Lyapunov's first method, explains asymptotic stability of 
the two equilibria \cite{Keelinfdiseasbook09,KuniyaJac18}. 
Drawing phase portraits has also visualized the behavior
outside the sufficiently small neighborhood of each equilibrium \cite{HETHinfdiseas}. 
For systematic analysis outside the small neighborhood, 
many studies constructed Lyapunov functions to invoke 
Lyapunov's second method for the SIR model 
and its variants 
(see \cite{KOROLyap02,KOROLyap04,ChenLyap04,FALLIDlypu07,OREGAN2010446,EnaNakIDlyapdelay11,NakEnaIDlyap14,SHUAIIDlyapu13,ChenLyapDI14,ElazIFDvacc19} 
and references therein). 
However, it has not been successful satisfactorily for next steps. 
Unless reasonable sublevel sets of constructed Lyapunov functions are confirmed, 
computing negative derivative of the functions along the 
trajectories cannot go beyond the local analysis 
Jacobian linearization offers. 
Sublevel sets are the only means to estimate of the domain of attraction 
in Lyapunov's second method \cite{Khalil_Book_02}.

Since achieving the negative derivative in reasonably large sublevel sets 
has been too hard for the SIR model, many preceding studies invoke 
LaSalle's invariance principle to relax the negativity into 
non-positivity \cite{Khalil_Book_02}. 
To use LaSalle's invariance principle, 
the notable study \cite{KOROLyap02} proposed to use a simplified model in which 
the newborn rate is endogenously determined to keep precise 
conservation of the total population. The key is that the simplification 
reduces the dimension of the system, and leads to an one-dimensional 
subspace for which the argument of LaSalle's invariance principle is 
effective since oscillation are not possible. 
The approach has facilitated the use of Lyapunov functions 
in infectious diseases widely 
(see, e.g., \cite{KOROLyap04,KOROgennonID06,EnaNakIDlyapdelay11} to name a few). 
However, it remains true that 
the simplifying assumption limits the use of models in prediction 
and understanding the disease transmission. 
In fact, the simplification ignores not only the actual newborn rate 
and its perturbation, but also individuals entering the area. 
Furthermore, LaSalle's invariance principle is invalid in the 
presence of time-varying parameters. Indeed, 
the non-positivity of the derivative does not have margins to 
accommodate perturbations and external fluxes. 
Strict negativity of the derivative is useful, and such Lyapunov 
functions are called strict Lyapunov functions \cite{MalFredStrLyapBook09}. 
The first objective of this paper is to construct a strict Lyapunov function 
for the SIR model without the simplification and the invariance principle, 
and to investigate its sublevel sets for understanding the attractivity 
behavior of the two equilibria on the entire state space. 

The second objective is to demonstrate robustness of the SIR model. 
Since the SIR model is nonlinear, asymptotic stability does not guarantee 
anything about behavior of trajectories in the presence of 
the variation of external parameters or disturbances \cite{Khalil_Book_02}. 
This paper employs the notion of input-to-state stability (ISS) 
to evaluate robustness of the SIR model with respect to perturbation 
of the newborn/immigration rate \cite{SONCOP}. 
To the best of the authors' knowledge, 
this ISS property has not been investigated for models of infectious 
diseases. 
Here, the perturbation input is for neither control input nor an operating variable. 
The word ``input'' originates from the terminology ``input-to-state stability'' 
which is a concept widely used in the field of nonlinear control systems 
\cite{Khalil_Book_02,MalFredStrLyapBook09}. 
The ``input'' represents uncertainty, parameter variation and disturbance. 
A nominal model is never perfect, In particular, in a real society, 
the newborn/immigration rate cannot always be maintained at a nominal value 
one wants to assume. 
To assess robustness with respect to that perturbation, 
this paper constructs functions called ISS Lyapunov functions \cite{SONISSV}. 
As a matter of fact, this construction leads to an answer to the first 
objective. When the newborn/immigration rate is constant, 
the ISS property reduces to the asymptotic stability. 
The constructed Lyapunov functions have negative derivative, and they address 
external variations by getting rid of LaSalle's invariance principle. 
Recall that the SIR model has two equilibria, and 
a bifurcation occurs as the newborn/immigration rate changes.
The paper demonstrates that the bifurcation takes place as 
a continuous change of the transient and the steady state with respect to the change of 
the newborn/immigration rate. The bifurcation is not a discontinuous phenomenon. 
This is true in both directions, from the disease-free equilibrium to 
the endemic equilibrium, and vice versa. 

\section{Preliminaries}\label{sec:symbols}

This paper uses the symbols $\Rset:=(-\infty,\infty)$, 
$\Rset_+:=[0,\infty)$ and $\Rset_+^n:=[0,\infty)^n$. 
For $v\in \Rset^n$, the symbol $|v|$ denotes a norm
which is selected consistently throughout the paper. 
It is the absolute value if $n=1$. 
This paper writes $\Gamma\in\calP$ if $\Gamma : \Rset_+\to\Rset_+$ is continuous  
and satisfies $\Gamma(0)=0$ and $\Gamma(s)>0$ for 
all $s\in\Rset_+\setminus\{0\}$. 
A function $\Gamma\in\calP$ is said to be of class $\calK$ and 
written as $\Gamma\in\calK$ if it is strictly increasing. 
A class $\calK$ function is said to be 
of class $\calK_{\infty}$ if it is unbounded. 
A continuous function $\Phi : \Rset_+\times\Rset_+\to\Rset_+$ is 
said to be of class $\calKL$ if, for each fixed $t\geq 0$, 
$\Phi(\cdot,t)$ is of class $\calK$ and, 
for each fixed $s>0$, $\Phi(s,\cdot)$ is 
decreasing and $\lim_{t\to\infty}\Phi(s,t)=0$. 
The zero function of appropriate dimension is denoted by $0$. 
Composition of the functions $\Gamma_1, \Gamma_2: \Rset\to\Rset$ is expressed as 
$\Gamma_1\circ\Gamma_2$. 
For a continuous function $f: \Rset^n\times\Rset^p\to\Rset^n$ satisfying 
$f(0,0)=0$, a system of the form 
\begin{align}
\dot{x}(t)=f(x(t),u(t)) 
\label{eq:notationsys}
\end{align}
is said to be input-to-state stable (ISS) with respect to the input $u$ 
\cite{SONCOP} if there exist $\Phi\in\calKL$ and $\Gamma\in{\calK}\cup\{0\}$ 
such that, for all continuous functions 
$u:\Rset_+\to\Rset^p$, all $x(0)\in\Rset^n$ and all $t\ge 0$, its unique solution 
$x(t)$ exists and satisfies
\begin{align}
\forall t\in\Rset_+\hspace{1.5ex}
|x(t)|\le
\Phi(|x(0)|, t) + \Gamma({\esssup}_{t\in\Rset_+}|u(t)|) . 
\label{eq:defISS}
\end{align}
The function $\Gamma$ is called an ISS-gain function. 
ISS of \eqref{eq:notationsys} implies globally asymptotic stability of the 
equilibrium $x=0$ for $u=0$. 
If a radially unbounded and continuously differentiable function 
$V: \Rset^n\to\Rset_+$ satisfies 
\begin{align}
&
\forall x\in\Rset^n\hspace{1.5ex}
\forall u\in\Rset^p\hspace{1.5ex}
\nonumber\\[-.3ex]
&
\hspace{4ex}
V(x)\ge \chi(|u|) 
\ \Rightarrow\ 
\frac{\partial V}{\partial x}(x)f(x,u)\le -\alpha(V(x))
\label{eq:ISSLyap}
\end{align}
for some $\chi\in\calK$ and some $\alpha\in\calP$, 
the function $V(x)$ is said to be an ISS Lyapunov function\footnote{
The original definition in \cite{SONISSV} employs 
$\alpha\in\calK$. However, the function $V$ can always be 
rescaled to modify $\alpha\in\calP$ into a class $\calK$ function.}. 
The existence of an ISS Lyapunov function 
guarantees ISS of system \eqref{eq:notationsys} \cite{SONISSV}. 
An ISS-gain function in \eqref{eq:defISS} is obtained as 
$\Gamma=\underline{\alpha}^{-1}\circ\chi$, where 
$\underline{\alpha}$ is a class $\calK_\infty$ function satisfying 
$\underline{\alpha}(|x|)\le V(x)$ for all $x\in\Rset^n$. 
ISS Lyapunov functions become conventional Lyapunov functions 
when $u=0$. 
All the above are standard definitions given for sign-indefinite system \eqref{eq:notationsys}. 
When the vector field $f$ generates only non-negative $x(t)$ in 
\eqref{eq:notationsys} defined with $x(0)\in\Rset_+^n$ and $u(t)\in\Rset_+^p$, all the above 
definitions and facts are valid 
by replacing $\Rset$ with $\Rset_+$.


For scalar $u$, one can define ISS with respect to the input $u(t)$ 
restricted to a range $(-\underline{u},\overline{u})$ for some 
constants $\underline{u}$ ,$\overline{u}\in\Rset_+\cup\{\infty\}$. 
To assess such ISS, one can just introduce a bijective function 
$\zeta: \Rset\to (-\underline{u},\overline{u})$ in $f(x,u)$ as 
$f(x,\zeta(r))$, where $\zeta(0)=0$. 
The standard restriction-free characterization \eqref{eq:ISSLyap} 
can be applied to $f(x,\zeta(r))$ with the auxiliary non-restricted input $r$. 
In this paper, for a compact set $\Omega\in\Rset^n$ satisfying $0\in\Omega$, 
system \eqref{eq:notationsys} is said to be ISS on the set $\Omega$ 
with respect to the input $u$ satisfying 
$u(t)\in(-\underline{u},\overline{u})$ if 
\eqref{eq:defISS} holds for all $x(0)\in\Omega$ and all 
$u(t)\in(-\underline{u},\overline{u})$. 
To measure the magnitude of $x$, the implication 
\eqref{eq:ISSLyap} employs $V(x)$ instead of $|x|$. Hence, 
ISS on $\Omega$ is implied by \eqref{eq:ISSLyap} if 
$x\in\Rset^n$ in \eqref{eq:ISSLyap} is replaced with a sublevel set 
\begin{align}
\overline{\Omega}(L):=\left\{x\in\Rset^n : L\ge V(x)\right\}
\end{align}
containing $\Omega$ and satisfying $L\ge \chi(|u|)$ for all $u$. 


If the function $V$ is not continuously differentiable, but 
locally Lipschitz, 
${\partial V}/{\partial x}\cdot f$ in \eqref{eq:ISSLyap} is replaced by 
\begin{align}
D^+V(x,u):=\liminf_{t \rightarrow 0+}
\frac{(V(\psi(t,x,u))-V(x))}{t}, 
\label{eq:defdini}
\end{align}
where $\psi(t,x,u)$ is the solution of \eqref{eq:notationsys}
with the initial condition $x$ and the input function $u$. 
If one writes it explicitly, 
\begin{align}
&
\forall x\in\Rset^n\hspace{1.5ex}
\forall u(0)\in\Rset^p\hspace{1.5ex}
\nonumber\\[-.3ex]
&
\hspace{4ex}
V(x)\ge \chi(|u(0)|) 
\ \Rightarrow\ 
D^+V(x,u) \le -\alpha(V(x)) .
\label{eq:ISSLyapDeni}
\end{align}
Let $\calN$ denote the subset of $\Rset^n$ where the gradient 
${\partial V}/{\partial x}$ does not exist. 
Rademacher's theorem shows that the set $\calN$ has measure 
zero for a locally Lipschitz $V$. Furthermore, 
the lower Dini derivative 
\eqref{eq:defdini} for each fixed $u$ agrees with 
$({\partial V}/{\partial x})f$ except in $\calN$. 
The existence of an ISS Lyapunov function defined with 
\eqref{eq:defdini} guarantees ISS of system \eqref{eq:notationsys} 
since $f$ and $\alpha$ continuous functions \cite{BACROSLiapbook05}. 

\begin{remark}
This paper demonstrates ISS of an epidemic model. Here, 
it is worth recalling that 
for nonlinear systems, global asymptotic stability of an equilibrium 
cannot guarantee boundedness of the state with respect to input of 
bounded magnitude \cite{Khalil_Book_02}. In fact, for example, the origin $I=0$ of 
$I$-system in \eqref{eq:sirI} is globally asymptotic stable 
for the nil input $S=0$, while the constant input 
$S>(\gamma+\mu)/\beta$ makes $I(t)$ unbounded. Therefore, 
$I$-system \eqref{eq:sirI} is not ISS. 
\end{remark}

\section{SIR Model}\label{sec:sir}

Let $x(t):=[S(t), I(t), R(t)]^T\in\Rset_+^3$ and assume that it satisfies 
\begin{subequations}\label{eq:sir}
\begin{align}
\dot{S}=&B -\mu S -\beta IS
\label{eq:sirS}
\\
\dot{I}=&\beta IS -\gamma I-\mu I
\label{eq:sirI}
\\
\dot{R}=&\gamma I-\mu R
\label{eq:sirR}
\end{align}
\end{subequations}
defined for any $x(0):=[S(0), I(0), R(0)]^T\in\Rset_+^3$ and 
any continuous function $B: \Rset_+\to\Rset_+$. 
In fact, for each $x(0)$ and $B$, the equation \eqref{eq:sir} admits a 
unique maximal solution $x(t)$ \cite{Khalil_Book_02}.  
Equation \eqref{eq:sir}, which is expressed compactly as 
\begin{align}\label{eq:sir_xform}
\dot{x}(t)=f(x(t),u(t))  
\end{align}
with the vector field $f=[f_1,f_2,f_3]^T$ and the input $u=B$, also guarantees 
$x_i(t)\ge 0$, $t\in\Rset_+$, 
for each $i=1,2,3$ since $f_i(x,u)\ge 0$ holds at $x_i=0$ for each $i=1,2,3$. 
The variables $x_1(t)$, $x_2(t)$ and $x_3(t)$ are denoted by 
$S(t)$, $I(t)$ and $R(t)$, respectively, since \eqref{eq:sir} is the equation 
popular model called SIR model (with demography) for infectious diseases 
\cite{Dietz67,KermackEpidemics27,Keelinfdiseasbook09}. 
The variable $S(t)$ describes the (continuum) number of the susceptible population, 
$I(t)$ is that of the infected population, while 
$R(t)$ is of the population recovered with immunity. 
The variable $B(t)$ is the newborn/immigration rate.  
The positive numbers $\beta$, $\gamma$ and $\mu$ are 
the transmission rate, the recovery rate and the death rate, respectively. 
Define the total population $N(t):=S(t)+I(t)+R(t)$ as usual. Since 
\begin{align}
\dot{N}(t)= B-\mu N(t)
\label{eq:issofn}
\end{align}
follows from \eqref{eq:sir}, $x(t)$ exists for all $t\in\Rset_+$, which 
is referred to the forward completeness of system \eqref{eq:sir_xform}. 
Property \eqref{eq:issofn} also implies that 
system \eqref{eq:sir_xform} is ISS with respect to 
the input $u$ \cite{HIinfectd20iiss}. Indeed, it is easy to see that 
\begin{align}
S(t)+I(t)+R(t)
&\le e^{-t}\left(S(0)+I(0)+R(0)-\frac{\overline{B}}{\mu}\right)+\frac{\overline{B}}{\mu}
\nonumber\\
&\le e^{-t}(S(0)+I(0)+R(0))+\frac{\overline{B}}{\mu}
\label{eq:Nest}
\end{align}
for all $t\in\Rset_+$ with respect to any $B(t)\in[0,\overline{B}]$. 
As discussed in \cite{HIinfectd20iiss}, 
$I$-system \eqref{eq:sirI} is not ISS with respect to its input $S$. 
The absence of ISS is characterized there as 
strong integral input-to-state stability on which 
this paper does not go into detail \cite{SontagSCL98,MIRITObil16,CHAANGITOStrISS14}. 
Interestingly, the absence of ISS of $I$-system provides a bifurcation 
selecting one of the two equilibria $x_e$ and $x_f$ depending 
on $\hat{R}_0$ to be explained below. $S$-system \eqref{eq:sirS} compensates 
the weak stability of $I$-system so that 
the overall system \eqref{eq:sir} is ISS.

Clearly, if the newborn/immigration rate is constant, i.e., 
$B(t)\equiv \hat{B}\ge 0$. equation \eqref{eq:sir} has two equilibria
\begin{align}
&
x_f:=\left[\frac{\hat{B}}{\mu}, 0, 0\right]^T
\label{eq:xf}
\\
&
x_e:=\left[
\frac{\gamma+\mu}{\beta}, 
\frac{\mu(\hat{R}_0-1)}{\beta}, 
\frac{\gamma(\hat{R}_0-1)}{\beta}
\right]^T, 
\label{eq:xe}
\end{align}
where the non-negative number 
\begin{align}
\hat{R}_0:=
\frac{\beta \hat{B}}{\mu(\gamma+\mu)}
\label{eq:R0hat}
\end{align}
is called the basic reproduction number \cite{Keelinfdiseasbook09}. 
The former state $x_f$ is called the disease-free equilibrium, while
the latter $x_e$ is called the endemic equilibrium. 
When $\hat{R}_0<1$, the endemic equilibrium $x_e$ disappears since 
$x(t)\in\Rset_+^3$. For $\hat{R}_0=1$, $x_e$ coincides with $x_f$. 
By local analysis based on Jacobian linearization\footnote{In the 
field of nonlinear systems and control, the term ``local'' is used 
exclusively for the existence of a sufficiently small set in which 
a claimed property holds true. One cannot specify the set a priori.}, 
the disease-free equilibrium $x_f$ is asymptotically stable
if $\hat{R}_0\le 1$ for the constant $B(t)\equiv \hat{B}\ge 0$ 
(see, e.g., \cite{Keelinfdiseasbook09}). 
The endemic equilibrium $x_e$ is asymptotically stable
if $\hat{R}_0> 1$. Here, as in the fundamental of stability theory, 
the proved asymptotic stability is 
local in the sense that the estimated domain of attraction is a 
sufficiently small neighborhood of the equilibrium. 
Construction of a Lyapunov function has a potential to go beyond the 
local property \cite{Khalil_Book_02}. 
If a Lyapunov function is found, an appropriate sublevel set of the 
function can be an estimate of the domain of attraction.

Once one of $x_f$ and $x_e$ is chosen as the target equilibrium. 
let $\hat{x}\in\Rset_+^3$ denote the chosen equilibrium and define 
\begin{align}
&
\tilde{x}(t):=x(t)-\hat{x}
\\
&
\tilde{u}(t):=B(t)-\hat{B}. 
\end{align}
Then the SIR model \eqref{eq:sir} can be rewritten as
\begin{align}\label{eq:sir_xformtil}
\dot{\tilde{x}}=\tilde{f}(\tilde{x},\tilde{u}),  
\end{align}
where the function $\tilde{f}=[\tilde{f}_1,\tilde{f}_2,\tilde{f}_3]^T$ 
satisfies $\tilde{f}(0,0)=0$. 
For brevity, let $[-\hat{x}_i,\infty)^3$
denote $[-\hat{x}_1,\infty)\times[-\hat{x}_2,\infty)\times[-\hat{x}_3,\infty)$. 
System \eqref{eq:sir_xformtil} is defined on $[-\hat{x}_i,\infty)^3$. 
The main objective of this paper is to prove that 
system \eqref{eq:sir_xformtil} is ISS with respect to 
the newborn/immigration rate perturbation $\tilde{u}$ on the entire 
state space of $\tilde{x}$. 
This property is not obvious from \eqref{eq:Nest} since ISS requires 
not only boundedness of the state $\tilde{x}$, but also 
a gain function that characterizes the boundedness as 
a continuous function $\Gamma$ of the input $\tilde{u}$ so that 
asymptotic stability is included as a special case, i.e., \eqref{eq:defISS}. 
Importantly, another major objective is the construction of 
an ISS Lyapunov function which serves as an classical (but, strict) 
Lyapunov function when $\tilde{u}=0$. 

\begin{remark}
This paper does not introduce assumptions on $B$ to make the analysis 
simple. For example, if $B=\mu(S+I+R)$ or an equivalent formulation 
is assumed, we have $S(t)+I(t)+R(t)=N$ for all $t\in\Rset_+$ with a positive constant $N$ 
\cite{Keelinfdiseasbook09}. This 
dependence between variables allows one to remove one of the three variable 
from \eqref{eq:sir}. 
Many analytical studies assume this simplification 
(e.g., \cite{KOROLyap02,KOROgennonID06,OREGAN2010446}), and 
the equation is sometimes called the SIRS model. 
The same implication has also been employed for variants 
of the SIR model in some studies (e.g., \cite{LIMULDseirlyapu95,KOROLyap04,EnaNakIDlyapdelay11,TianNetSIRLyap20}). 
The assumption allows us to understand basic mechanisms of disease models clearly, and 
the aforementioned studies have provided a lot of important observations we now rely on. 
Nevertheless, the simplification disallows one to consider perturbation of birth and 
immigration, and $B$ becomes endogenous. 
The simplification prevents the robustness analysis on which this paper focuses. 
\end{remark}

\section{Disease-Free Equilibrium}\label{sec:lyapxf}

The first result in this paper is stated as the next theorem. 

\begin{theorem}\label{thm:sirf}
Suppose that $\hat{B}>0$ and 
\begin{align}
\hat{R}_0< 1 
\label{eq:R0bifurtoxf}
\end{align}
hold. Let $\hat{x}=x_f$. 
Then the disease-free equilibrium $\tilde{x}=0$ of the SIR model \eqref{eq:sir} 
is asymptotically stable, and the set $[-\hat{x}_1,\infty)\times\Rset_+^2$ is 
the domain of attraction. Moreover, 
the SIR model \eqref{eq:sir} is ISS on $[-\hat{x}_1,\infty)\times\Rset_+^2$ 
with respect to the newborn rate 
perturbation $\tilde{u}$ satisfying 
\begin{align}
\forall t\in\Rset_+\hspace{1.5ex}
\tilde{u}(t)\in [-\hat{B},\infty) . 
\label{eq:Brange}
\end{align}
Furthermore, the function 
$\tilde{V}: \Rset\times\Rset_+^2\to\Rset_+$ defined by 
\begin{align}
&
\tilde{V}(\tilde{x})\!=\hspace{-.5ex}\left\{\begin{array}{ll}
-\dfrac{\mu_0\tilde{x}_1}{\beta \hat{x}_1}
&
\tilde{x}_1<
-\dfrac{\beta\hat{x}_1}{\mu_0}(\tilde{x}_2\!+\!\lambda_3\tilde{x}_3)
\\[1.8ex]
\tilde{x}_2+\lambda_3\tilde{x}_3,
&
-\dfrac{\beta\hat{x}_1}{\mu_0}(\tilde{x}_2\!+\!\lambda_3\tilde{x}_3)
\le \tilde{x}_1< 0 
\\[1.8ex]
\tilde{x}_1+\tilde{x}_2+\lambda_3\tilde{x}_3, 
&
0 \le \tilde{x}_1
\end{array}\right.
\label{eq:defVf}
\\
&
0<\mu_0<\mu
\label{eq:mu0}
\\
&
\max\left\{\frac{\mu}{\gamma+\mu}-\hat{R}_0,\, 0\right\}<
\epsilon <1-\hat{R}_0
\label{eq:epsi}
\\
&
\gamma_0=
(\gamma+\mu)(\hat{R}_0+\epsilon)-\mu
\label{eq:gam0}
\\
&
\lambda_3=1-\frac{\gamma_0}{\gamma}
\label{eq:lamsirf}
\end{align}
is locally Lipschitz on $\Rset\times\Rset_+^2$, 
and an ISS Lyapunov function on $[-\hat{x}_1,\infty)\times\Rset_+^2$ 
with respect to \eqref{eq:Brange}. 
\end{theorem}
\begin{proof} 
First, recall that $\hat{B}>0$ and \eqref{eq:R0bifurtoxf} imply 
$\hat{x}_1=\hat{B}/\mu>0$, $\hat{x}_2=\hat{x}_3=0$. Thus, 
$\tilde{x}_2=x_2\ge 0$, $\tilde{x}_3=x_3\ge 0$. 
Definition \eqref{eq:gam0} and 
conditions \eqref{eq:R0bifurtoxf} and \eqref{eq:epsi} yield 
$\hat{R}_0+\epsilon<1$, and $0< \gamma_0< \gamma$. Hence, 
$\lambda_3>0$ in \eqref{eq:lamsirf}. 
Define 
\begin{subequations}\label{eq:defABC}
\begin{align}
\bfA&:=\left\{ \tilde{x}\!\in\!\Rset\times\Rset_+^2 : 
0\le \tilde{x}_1\right\}
\label{eq:defABC_A}
\\
\bfB&:=\left\{ \tilde{x}\!\in\!\Rset\times\Rset_+^2 : 
-\dfrac{\beta\hat{x}_1}{\mu_0}(\tilde{x}_2\!+\!\lambda_3\tilde{x}_3)
\le \tilde{x}_1< 0\right\}
\label{eq:defABC_B}\\
\bfC&:=\left\{ \tilde{x}\!\in\!\Rset\times\Rset_+^2 : 
\tilde{x}_1<
-\dfrac{\beta\hat{x}_1}{\mu_0}(\tilde{x}_2\!+\!\lambda_3\tilde{x}_3)
\right\} . 
\label{eq:defABC_C}
\end{align}
\end{subequations}
The partitioning of 
\eqref{eq:defABC} clearly satisfies 
$\bfA\cup \bfB\cup \bfC=\Rset\times\Rset_+^2$. 
By definition \eqref{eq:defVf}, $\tilde{V}(\tilde{x})=0$ holds if and 
only if $\tilde{x}=0$ in $\Rset\times\Rset_+^2$. 
We have $\tilde{V}(\tilde{x})>0$ for all 
$\Rset\times\Rset_+^2\}\setminus\{0\}$ 
At $\tilde{x}_1=0$, the function $\tilde{V}$ is continuous and 
$\tilde{V}(\tilde{x}_1)=\tilde{x}_2+\lambda_3\tilde{x}_3$. 
It is also verified at the point 
$-{\beta\hat{x}_1}(\tilde{x}_2\!+\!\lambda_3\tilde{x}_3)/{\mu_0}$ that 
the function $\tilde{V}$ is continuous and 
$\tilde{V}(\tilde{x}_1)=\tilde{x}_2+\lambda_3\tilde{x}_3$. 
Hence, the function $\tilde{V}$ defined by \eqref{eq:defVf} is locally Lipschitz 
on $\Rset\times\Rset_+^2$. 
Now, we evaluate the derivative of $\tilde{V}(\tilde{x}(t))$ along the solution 
$\tilde{x}(t)$ of \eqref{eq:sir} region by 
region\footnote{
Since the derivative of $\tilde{V}$ is defined except on the boundaries between the regions, 
the derivative is computed except on those boundaries. 
Alternatively, if one considers the isolated segment of $\tilde{V}$ defined in 
each region separately, the derivative is defined on the boundaries. 
Both evaluations are valid since the upper bounds to be obtained are 
continuous in individual regions, and a common continuous upper bound will be 
derived later at \eqref{eq:DVABC}. Note that the common bound prevents each bound 
from approaching zero on the boundaries. 
This argument also applies to the proof of Theorem \ref{thm:sireLyap}. 
}.
In region $\bfA$, by virtue of 
$\hat{B} -\mu \hat{x}_1 -\beta \hat{x}_2\hat{x}_1
=\hat{B} -\mu \hat{x}_1=0$, from \eqref{eq:lamsirf} 
we obtain 
\begin{align}\label{eq:DVA}
\dfrac{\partial \tilde{V}}{\partial \tilde{x}}\tilde{f}
&= 
B-\mu S -\beta IS
+\beta IS -\gamma I-\mu I
+\lambda_3(\gamma I-\mu R)
\nonumber\\
&=\tilde{u}
-\mu\tilde{x}_1 -(\gamma_0+\mu)\tilde{x}_2 -\lambda_3\mu\tilde{x}_3  
\nonumber\\
&\le 
-\mu\tilde{V}(\tilde{x})+\tilde{u} . 
\end{align}
In region $\bfB$ we have 
\begin{align*}
\beta x_1<\beta \hat{x}_1=\frac{\beta \hat{B}}{\mu}=(\gamma+\mu)\hat{R}_0 . 
\end{align*}
Due to \eqref{eq:gam0}, in region $\bfB$, 
\begin{align}\label{eq:DVB}
\dfrac{\partial \tilde{V}}{\partial \tilde{x}}\tilde{f}
&=
\beta IS -\gamma I-\mu I
+\lambda_3(\gamma I-\mu R)
\nonumber\\
&=-(\gamma_0+\mu-\beta x_1) x_2-\lambda_3\mu x_3
\nonumber\\
&\le-\epsilon (\gamma_0+\mu)\tilde{x}_2-\lambda_3\mu\tilde{x}_3 
\nonumber\\
&\le 
-\underline{\epsilon}\tilde{V}(\tilde{x})
\end{align}
is obtained, where $\underline{\epsilon}:=\min\{\epsilon(\gamma_0+\mu),\mu\}$. 
In region $\bfC$, since the definition of $\bfC$ yields 
\begin{align*}
\tilde{x}_1<
-\frac{\beta\hat{x}_1}{\mu_0}(x_2+\lambda_3x_3)
\le
-\frac{\beta\hat{x}_1}{\mu_0}x_2
<
-\frac{\beta}{\mu_0}x_1x_2
\end{align*}
it is verified that 
\begin{align}\label{eq:DVC}
\dfrac{\partial \tilde{V}}{\partial \tilde{x}}\tilde{f}
&= 
-\frac{\mu_0}{\beta \hat{x}_1}
(B-\mu S -\beta SI)
\nonumber\\
&=
-\frac{\mu_0}{\beta \hat{x}_1}
(\tilde{u}-(\mu-\mu_0)\tilde{x}_1-\mu_0\tilde{x}_1 -\beta x_1x_2)
\nonumber\\
&<
\frac{\mu_0}{\beta \hat{x}_1}
\left(
(\mu-\mu_0)\tilde{x}_1-\tilde{u}
\right). 
\end{align}
Note that $\tilde{x}_1<0$ in $\bfC$.

Due to \eqref{eq:mu0}, 
combining \eqref{eq:DVA}, \eqref{eq:DVB} and \eqref{eq:DVC}, for 
an arbitrarily given $\delta\in(0,1)$, 
we obtain 
\begin{align}\label{eq:DVABC}
&
\tilde{V}(\tilde{x})\ge \frac{1}{\delta(\mu-\mu_0)}|\tilde{u}|
\ \Rightarrow\ 
\dfrac{\partial \tilde{V}}{\partial \tilde{x}}\tilde{f}
\le -(1-\delta)(\mu-\mu_0)\tilde{V}(\tilde{x})
\end{align}
for all $\tilde{x}\in \Rset\times\Rset_+^2$ and 
$\tilde{u}\in[-\hat{B},\infty)$ 
except on the boundaries between the regions. 
The continuity of involved functions and Rademacher's theorem allow 
\eqref{eq:DVABC} to hold for all $\tilde{x}\in \Rset\times\Rset_+^2$ by 
replacing the derivative with the lower Dini derivative. 
Equation \eqref{eq:sir} by itself guarantees 
$\tilde{x}(t)\in[-\hat{x}_1,\infty)\times\Rset_+^2$
for all $t\in\Rset_+$  
with respect to all $\tilde{x}(0)\in[-\hat{x}_1,\infty)\times\Rset_+^2$ 
and $\tilde{u}\in[-\hat{B},\infty)$.  
Therefore, all the claims are proved. 
\end{proof} 


Let the perturbed basic reproduction number $R_0(t)$ be 
defined with $B=\tilde{u}(t)+\hat{B}$, while 
the (nominal) basic reproduction number $\hat{R}_0$ has been 
defined with the nominal rate $\hat{B}$. 
If $\limsup_{t\to\infty}R_0(t)>1$, 
the state $x(t)$ does not converge to $x_f$
even for \eqref{eq:R0bifurtoxf}. 
In the same way, if $\liminf_{t\to\infty}R_0(t)<1 $ holds, 
the state $x$ does not converge to $x_e$ 
even for \eqref{eq:R0bifurtoxe} to be presented in the next section. 
The established ISS property does not override the mechanism of 
the basic reproduction number. 
Note that $\hat{R}_0=1$ holds if and only if $x_f=x_e$. 
The ISS property obtained in Theorem \ref{thm:sirf} not only guarantees 
the boundedness of $\tilde{x}(t)$ with respect to bounded $\tilde{u}(t)$, 
but also continuous variation of the bound with respect to the 
maximum magnitude of $\tilde{u}(t)$. 
Interestingly, the continuous transition holds true although 
the change of $\tilde{u}$ causes a bifurcation. 
The obtained property \eqref{eq:DVABC} together with definition \eqref{eq:defVf} 
establishes that the bound of the state variable $\tilde{x}$
is a linear function of the magnitude of the variation $\tilde{u}$. 
Figure \ref{fig:vf} illustrates level sets of the ISS Lyapunov function 
\eqref{eq:defVf} for $\tilde{V}=10$, $30$, $60$, $100$, $180$, $260$, 
$340$, ..., $500$. 
The parameters are 
$\beta=0.0002$, $\mu=0.015$, $\gamma=0.032$, $\hat{B}=3$ and 
$\mu_0=0.0149$, and they satisfy $\hat{R}_0 = 0.851<1$, 
\eqref{eq:mu0} and \eqref{eq:epsi} with 
$\epsilon=0.0745$ and $\mu_0 = 0.0148$.


\begin{remark}
The preceding study \cite{HIinfectd20iiss} demonstrated 
ISS of the SIR model \eqref{eq:sir} irrespective of the value $R_0$
by treating the entire amount $B$ as the input of the ISS property. 
It means that in \cite{HIinfectd20iiss}, the whole $R_0$ is a 
disturbance, and its nominal value is $\hat{R}_0=0$. 
Hence, the focused equilibrium was $\hat{x}=[0, 0 ,0]^T$ in 
\cite{HIinfectd20iiss}, instead of $x_f=[\hat{B}_0/\mu, 0 ,0]^T$. 
The ISS of \eqref{eq:sir} for $\hat{x}=0$ does not conclude 
that the state $x$ converges to 
the point $x_f=[\hat{B}_0/\mu, 0 ,0]^T$ when $B(t)\equiv \hat{B}$ 
and $\hat{R}_0<1$. 
The ISS property of \eqref{eq:sir} with respect to 
the non-zero equilibrium is not obvious from the ISS property 
of \eqref{eq:sir} with the zero equilibrium $x=\hat{x}=0$ either. 
\end{remark}

\begin{figure}[t]
\begin{center}
\includegraphics[width=8.5cm,height=5.2cm]{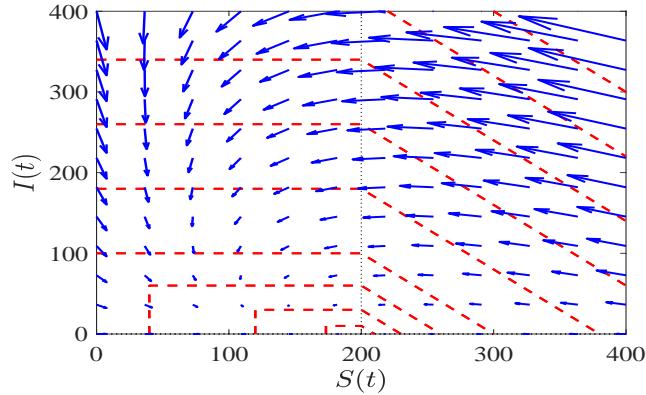}
\vspace{-1.2ex}
\end{center}
\caption{Level sets of the ISS Lyapunov function \eqref{eq:defVf} for the 
disease-free equilibrium with $\hat{B}= 3$ (Dash lines); 
$S$ and $I$ are $x_1$ and $x_2$, respectively; 
The arrows are segments of 
trajectories for $B(t)=\hat{B}$; The dotted line is $x_1=\hat{x}_1$.} 
\label{fig:vf}
\end{figure}

\section{Endemic Equilibrium}\label{sec:lyapxe}

This section constructs a Lyapunov function dealing with the endemic 
equilibrium $x_e$. For this purpose, we set $\hat{x}=x_e$. 
When the disease-free equilibrium $x_f$ was of interest, the component $x_2$ of 
the trajectories $x(t)$ of the SIR model \eqref{eq:sir} could not go below $\hat{x}_2$. 
Thus, the level contours of the Lyapunov function \eqref{eq:defVf} 
were sheared off at the plane of $x_2=\hat{x}_2$ in the three-dimensional space of $x$. 
Since $x_2$ can go below $\hat{x}_2$ for the endemic equilibrium $x_e$, 
an end of each level contour of the Lyapunov function \eqref{eq:defVf} needs to be 
placed more carefully at the plane of $x_2=0$ in order be able to connect the other 
end to form a loop. 
In addition to closing the contours, the influence of equation \eqref{eq:sirR} 
is not as simple as that in the case of the disease-free equilibrium. 
In fact, the endemic equilibrium also allows $x_3$ to be go below $\hat{x}_3$. 
The term of $x_2$ in \eqref{eq:sirR} needs to be taken care of 
depending on the sign of $\tilde{x}_2$ and $\tilde{x}_3$ 
to make the Lyapunov function decrease along the trajectory $x(t)$. 

Let $x_{1,2}(t)=[x_1(t), x_2(t)]^T$, and define 
\begin{align}
&
\Omega:=\left\{\tilde{x}\in[-\hat{x}_i,\infty)^3 : 
\tilde{x}_2\ne -\hat{x}_2
\right\}
\label{eq:defOmega}
\\
&
G_{1,2}:=\left\{\tilde{x}_{1,2}\in[-\hat{x}_i,\infty)^2 : 
\tilde{x}_1+\tilde{x}_2>-\hat{x}_2, \ \tilde{x}_2\ne -\hat{x}_2
\right\}
\label{eq:defG12}
\\
&
G:=G_{1,2}\times[-\hat{x}_3,\infty) .
\label{eq:defG}
\end{align}
The set $\Omega$ is the domain on which we want to establish stability 
properties. The situation $\tilde{x}_2=-\hat{x}_2$ is and must be removed from 
$\Omega$ since $x_f$ remains an equilibrium of the SIR model \eqref{eq:sir} 
independently of $\hat{B}$, i.e., $\hat{R}_0$. 
Indeed, the point $I=0$, i.e., $\tilde{x}_2=-\hat{x}_2$, remains an equilibrium
of \eqref{eq:sirI} irrespective of $x_1$ and $x_3$. 
The set $G$ is the domain on which a Lyapunov function is to be constructed. 
The following summarizes stability properties established in 
this section. 

\begin{theorem}\label{thm:sire}
Assume that $\hat{B}>0$ and 
\begin{align}
\hat{R}_0> \frac{\gamma}{\mu}+2 
\label{eq:R0bifurtoxe}
\end{align}
hold. Let $\hat{x}=x_e$. 
Then the endemic equilibrium $\tilde{x}=0$ of the SIR model \eqref{eq:sir} 
is asymptotically stable, and any compact subset in $\Omega$ belongs to
the domain of attraction. 
Furthermore, for an arbitrarily given compact set $\underline{G}$ 
contained in the interior of $G$, 
there exists a compact set 
$\overline{G}\supset\underline{G}$ such that 
the SIR model \eqref{eq:sir} is ISS on $\overline{G}$ 
with respect to the newborn/immigration rate perturbation $\tilde{u}$ satisfying 
\begin{align}
\forall t\in\Rset_+\hspace{1.5ex}
\tilde{u}(t)\in [-\hat{B},\infty)
\label{eq:issuall}
\end{align}
\end{theorem}

The above theorem is established by the construction of the following 
ISS Lyapunov function. 

\begin{theorem}\label{thm:sireLyap}
Assume that $\hat{B}>0$ and \eqref{eq:R0bifurtoxe} are satisfied. 
Define the function $V$ by 
\begin{align}
\tilde{V}(\tilde{x})=\tilde{V}_{1,2}(\tilde{x}_{1,2})+
\tilde{V}_3(\tilde{x}_3) 
\label{eq:defVe}
\end{align}
with 
\begin{align}
&
\tilde{V}_{1,2}(\tilde{x}_{1,2})=
\left\{\begin{array}{ll}
P^{-1}\left(-\lambda_1\tilde{x}_1+\hat{\lambda}_2\tilde{x}_2\right), 
&
\mbox{\small$0\le\tilde{x}_2,  \ 
\tilde{x}_1< \nu(\tilde{x}_2)
$}
\\
(\lambda_2-k\lambda_1)\tilde{x}_2,
&
\mbox{\small$0\le\tilde{x}_2,  \ 
\nu(\tilde{x}_2)
\le \tilde{x}_1< -k\tilde{x}_2
$}
\\
\lambda_1\tilde{x}_1+\lambda_2\tilde{x}_2, 
&
\mbox{\small$0\le\tilde{x}_2,  \ 
 -k\tilde{x}_2 \le \tilde{x}_1
$}
\\
P^{-1}\left(-\lambda_1\tilde{x}_1-\lambda_2\tilde{x}_2\right),     
&
\mbox{\small$\tilde{x}_2<0,  \ 
\tilde{x}_1\le -k\tilde{x}_2
$}
\\
P^{-1}\left((k\lambda_1-\lambda_2)\tilde{x}_2\right),     
&
\mbox{\small$\tilde{x}_2<0,  \ 
-k\tilde{x}_2<\tilde{x}_1 \le 
\theta^{-1}(-\tilde{x}_2)
$}
\\
\lambda_1\tilde{x}_1-\hat{\lambda}_2\tilde{x}_2, 
&
\mbox{\small$\tilde{x}_2<0,  \ 
\theta^{-1}(-\tilde{x}_2)<\tilde{x}_1
$}
\end{array}\right.
\label{eq:defVe12}
\\
&
\tilde{V}_3(\tilde{x}_3)=\lambda_3|\tilde{x}_3|
\label{eq:defVe3}
\end{align}
and
\begin{align}
&
\theta(s)=\hat{x}_2-\frac{\hat{x}_1\hat{x}_2}{\hat{x}_1+s}
, \quad s\in(-\hat{x}_1,\infty)
\label{eq:deftheta}
\\
&
\theta^{-1}(s)=\frac{\hat{x}_1\hat{x}_2}{\hat{x}_2-s}-\hat{x}_1
, \quad s\in(-\infty,\hat{x}_2)
\label{eq:defthetainv}
\\
&
0<\lambda_1=\lambda_2
\label{eq:lam12sire}
\\
&
0<k<\min\left\{
1-\frac{\gamma+\mu}{\mu(\hat{R}_0-1)},\, 
\frac{\lambda_2\theta^{-1}(-\overline{L}/\lambda_2)}{\lambda_1\theta^{-1}(-\overline{L}/\lambda_2)-\overline{L}}  
\right\}=k_0
\label{eq:ksire}
\\
&
0<\lambda_3<\min\biggl\{
\frac{k\mu\lambda_1(\hat{R}_0-1)(1-k)}{\gamma}, \ 
\frac{\hat{\lambda}_2^2}{(1-k)\lambda_1}, \, 
\nonumber\\
&\hspace{23.5ex}
\frac{\beta\hat{\lambda}_2
(\hat{x}_2-\theta\circ\omega^{-1}(\overline{L}))}{\gamma}, \, 
\hat{\lambda}_2
\biggr\}
\label{eq:lam3sire}
\\
&
\omega(s)=\lambda_1s+\hat{\lambda}_2\theta(s)
, \quad s\in(-\hat{x}_1,\infty)
\label{eq:defomega}
\\
&
P(s)=(\lambda_2-k\lambda_1)\theta\circ \omega^{-1}(s)
, \quad s\in\Rset
\label{eq:defP}
\\
&
P^{-1}(s)=\lambda_1\theta^{-1}\left(\dfrac{s}{\lambda_2-k\lambda_1}\right) 
+\dfrac{\hat{\lambda}_2s}{\lambda_2-k\lambda_1}, 
\quad 
s\in(-\infty,(\lambda_2-k\lambda_1)\hat{x}_2)
\label{eq:defPinv}
\\
&
\nu(s)=
\dfrac{1}{\lambda_1}
\left(\hat{\lambda}_2s-P((\lambda_2-k\lambda_1)s)\right)
, \quad s\in\Rset
\label{eq:defnu}
\end{align}
for $\overline{L}>0$ and $\hat{\lambda}_2>0$. 
If $\overline{L}>0$ and $\hat{\lambda}_2>0$ satisfy 
\begin{align}
\forall L\in[0,\overline{L}]\hspace{1.5ex}
\nu\left(\dfrac{L}{\lambda_2-k\lambda_1}\right)
\le \theta^{-1}\!\left(-\dfrac{L}{\lambda_2-k\lambda_1}\right) , 
\label{eq:TLcorneredgeout}
\end{align}
the function $\tilde{V}$ is locally Lipschitz on the set 
\begin{align}
H(\hat{\lambda}_2,k,\overline{L}):=\left\{ \tilde{x}\in\Rset^3 :\hspace{-3ex} 
\begin{array}{c}
-\hat{x}_2<\tilde{x}_2\le 
\dfrac{\overline{L}}{\lambda_2(1-k)}, \\
-\tilde{x}_1-\tilde{x}_2<(1-k)\hat{x}_2, \\
-\lambda_1\tilde{x}_1+\hat{\lambda}_2\tilde{x}_2<\lambda_2(1-k)\hat{x}_2
\end{array}
\hspace{-1ex}\right\}, 
\label{eq:defH}
\end{align}
and the function $\tilde{V}$ is an ISS Lyapunov function on 
\begin{align}
\overline{G}(\hat{\lambda}_2,k,\overline{L}):= \left\{\tilde{x}\in[-\hat{x}_i,\infty)^3 : 
\tilde{V}(\tilde{x})\le \overline{L} \right\}
\label{eq:sublevelset}
\end{align}
with respect to the input $\tilde{u}$ satisfying 
\begin{align}
\forall t\in\Rset_+\hspace{1.5ex}
\tilde{u}(t)\in 
\left(-\delta\mu\frac{P(\overline{L})}{\lambda_1}, \frac{\delta\mu\overline{L}}{\lambda_1}\right)  
\label{eq:issuL}
\end{align}
for an arbitrarily given $\delta\in(0,1)$. 
\end{theorem}

The parameter $\hat{\lambda}_2>0$ introduced in \eqref{eq:defVe12} 
copes with the both-sided variables $\tilde{x}_2$ and $\tilde{x}_3$. 
The following lemma shows that the sublevel sets of 
the ISS Lyapunov function $\tilde{V}$ can always 
cover\footnote{cover any bounded sets in}
the set $G$ entirely 
as $\overline{L}\to \infty$ and $\hat{\lambda}_2\to 0$, which is the key to 
the establishment of Theorem \ref{thm:sire} from 
Theorem \ref{thm:sireLyap}. In fact, it forms 
a central and unique idea of this paper. 

\begin{lemma}\label{lem:sublevel}
Assume that \eqref{eq:R0bifurtoxe} is satisfied. Suppose that 
\eqref{eq:defVe}-\eqref{eq:defnu} and \eqref{eq:defH}-\eqref{eq:sublevelset} 
are defined and given. Then the following hold true: 
\\*
\textbf{(i)\ }
For any compact set $\underline{G}$ contained in the interior of $G$, 
there exist $\overline{\lambda}_2>0$, $\overline{L}\ge 0$ and 
$\overline{k}\in(0,k_0)$ such that 
\eqref{eq:TLcorneredgeout} and 
\begin{align}
\underline{G}
\subset
\overline{G}(\hat{\lambda}_2,k,\overline{L}) 
\subset
G 
\label{eq:sublevelsetglobal}
\end{align}
are satisfied 
for all $\hat{\lambda}_2\in(0,\overline{\lambda}_2]$ and 
all $k\in(0,\overline{k}]$. 
\\*
\textbf{(ii)\ }
For each $k$ satisfying \eqref{eq:ksire}, 
\begin{align}
0 \le a \le b\ \Rightarrow\ 
\overline{G}(b,k,L)\subset 
\overline{G}(a,k,L)
\label{eq:GLorderinhatlam2}
\end{align}
holds for all $L\in[0,\overline{L}]$ 
if \eqref{eq:TLcorneredgeout} holds. 
\\*
\textbf{(iii)\ }
For each $\hat{\lambda}_2\ge 0$, 
\begin{align}
&
0 \le a \le b\ \Rightarrow\ 
\overline{G}(\hat{\lambda}_2,b,L)
\cap\left\{\!\tilde{x}_2\!\le\!\frac{L}{\lambda_2}\!\right\}\subset 
\overline{G}(\hat{\lambda}_2,a,L)\cap\left\{\!\tilde{x}_2\!\le\!\frac{L}{\lambda_2}\!\right\}
\label{eq:GLorderink}
\end{align}
holds for all $L\in[0,\overline{L}]$ 
if \eqref{eq:TLcorneredgeout} holds. 
\\*
\textbf{(iv)\ }
Property \eqref{eq:TLcorneredgeout} holds for any 
all $\overline{L}\in\Rset_+$ if $\overline{\lambda}_2=0$. 
\end{lemma}

As demonstrated in Remark \ref{rem:linISSgain} in Section \ref{sec:endprf}, 
the ISS-gain function from $\tilde{u}$ to $\tilde{x}$ is bounded from 
above by a linear function. Recall that if a negative value 
$\tilde{u}$ goes below the threshold determined by the basic 
reproduction number, a bifurcation occurs. 
The ISS property established by Theorem \ref{thm:sire} establishes 
a linear transition globally in spite of the bifurcation. 

Level sets of the ISS Lyapunov function \eqref{eq:defVe} are shown in 
Fig. \ref{fig:ve} for 
$\tilde{V}=20$, $100$, $180$, $260$, ..., $340$. 
The parameters are 
$\beta=0.0002$, $\mu=0.015$, $\gamma=0.032$  and $\hat{B}=17$, 
and they satisfy 
$\hat{R}_0 = 4.82271>4.1333={\gamma}/{\mu}+2$. 
It can be verified that 
$\hat{\lambda}_2=0.01$, $k=0.0902$ and $\overline{L}=340$ 
fulfill \eqref{eq:ksire} and \eqref{eq:TLcorneredgeout}. 
The level sets can be expanded further by using smaller 
$\hat{\lambda}_2$, $k$ and $1/\overline{L}$.  

\begin{remark}
In contrast to  $\tilde{x}_2=-\hat{x}_2$ which is an equilibrium
of \eqref{eq:sirI} irrespective of $x_1$ and $x_3$, 
the equilibrium of $x_1$-equation \eqref{eq:sirS} depends on its input $x_2$, and 
the equilibrium of $x_3$-equation \eqref{eq:sirR} is influenced by its input $x_2$. 
Therefore, excluding the two-dimensional spacs $\tilde{x}_1=-\hat{x}_1$ and 
$\tilde{x}_3=-\hat{x}_3$ from 
the domain of $\tilde{V}$ is not necessary. In fact, the function chosen in 
\eqref{eq:defVe} is not forced to be unbounded at $\tilde{x}_1=-\hat{x}_1$ 
and $\tilde{x}_3=-\hat{x}_3$. 
The popular logarithmic function \cite{KOROLyap02} excludes 
$\tilde{x}_1=-\hat{x}_1$ and and $\tilde{x}_3=-\hat{x}_3$, and 
becomes unbounded there. 
\end{remark}

\begin{figure}[t]
\begin{center}
\includegraphics[width=9.2cm,height=5.8cm]{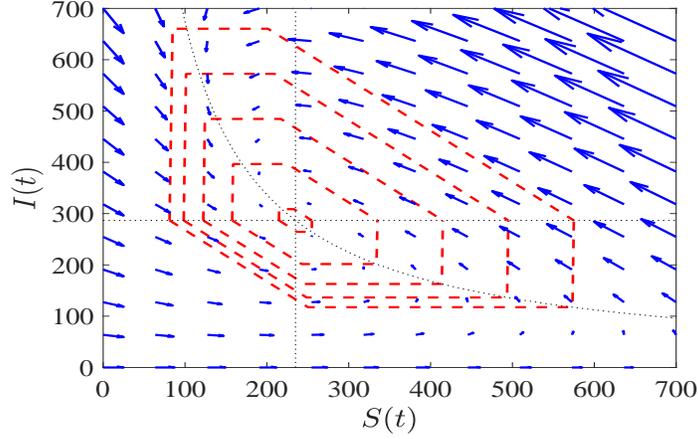}
\vspace{-1.2ex}
\end{center}
\caption{Level sets of the ISS Lyapunov function \eqref{eq:defVe} for the 
endemic equilibrium with $\hat{B}= 17$ (Dash lines); 
$S$ and $I$ are $x_1$ and $x_2$, respectively; 
The arrows are segments of 
trajectories for $B(t)=\hat{B}$; The dotted lines are $x_1=\hat{x}_1$, $x_2=\hat{x}_2$ and 
$x_1x_2=\hat{x}_1\hat{x}_2$; The lower left area along $x_1$-axis cannot be filled with 
sublevel sets of any Lyapunov functions.} 
\label{fig:ve}
\end{figure}

\section{Proofs for the Endemic Equilibrium}\label{sec:endprf}

\subsection{Proof of Lemma \ref{lem:sublevel}}\label{ssec:lemsublevel}

Since \eqref{eq:R0bifurtoxe} implies $({\gamma+\mu})/{\mu}<\hat{R}_0-1$, 
we have 
\begin{align*}
0<\frac{\gamma+\mu}{\mu(\hat{R}_0-1)}<1 
\end{align*}
and $\hat{x}_1< \hat{x}_2$. 
Property \eqref{eq:ksire} yields $k\in(0,1)$, and property 
\eqref{eq:lam12sire} guarantees $\lambda_2-k\lambda_1>0$. 
The choice \eqref{eq:ksire} yields 
\begin{align*}
\theta^{-1}\!\left(-\dfrac{\overline{L}}{\lambda_2}\right) 
< 
\dfrac{-\overline{L}}{\lambda_2/k-\lambda_1} .
\end{align*}
Since the function $\theta^{-1}$ satisfies $\theta^{-1}(0)=0$ and 
$(\theta^{-1})^\prime(s)>0$ for all $s\in(-\infty,\hat{x}_2)$, we have 
\begin{align*}
\theta^{-1}\!\left(-\dfrac{\overline{L}}{\lambda_2-k\lambda_1}\right) 
< 
\theta^{-1}\!\left(-\dfrac{\overline{L}}{\lambda_2}\right) . 
\end{align*}
From \eqref{eq:defthetainv} it is verified that 
$(\theta^{-1})^{\prime\prime}(s)>0$ holds  for all $s\in(-\infty,\hat{x}_2)$. 
Thus, 
\begin{align}
\forall L\in[0,\overline{L}]\hspace{1.5ex}
\theta^{-1}\!\left(-\dfrac{L}{\lambda_2-k\lambda_1}\right) 
\le \dfrac{-kL}{\lambda_2-k\lambda_1}  
\label{eq:TLcorneredgeslope}
\end{align}
is achieved. Combining this with \eqref{eq:TLcorneredgeout} yields 
\begin{align*}
\forall L\in[0,\overline{L}]\hspace{1.5ex}
\nu\left(\dfrac{L}{\lambda_2-k\lambda_1}\right)
\le \dfrac{-kL}{\lambda_2-k\lambda_1} . 
\end{align*}
Therefore, the partitioning in \eqref{eq:defVe12} 
is well-defined as long as $\tilde{x}\in H(\hat{\lambda}_2,k,\overline{L})$. 
%
%
Define 
\begin{subequations}\label{eq:defABCe}
\begin{align}
\bfA&:=\left\{ \tilde{x}\!\in\! H(\hat{\lambda}_2,k,\overline{L}) : 0\le\tilde{x}_2, \ 
-k\tilde{x}_2 \le \tilde{x}_1
\right\}
\label{eq:defABC_Ae}
\\
\bfB&:=\left\{ \tilde{x}\!\in\! H(\hat{\lambda}_2,k,\overline{L}) : 0\le\tilde{x}_2, \  
\nu(\tilde{x}_2)\le \tilde{x}_1< -k\tilde{x}_2
\right\}
\label{eq:defABC_Be}\\
\bfC&:=\left\{ \tilde{x}\!\in\! H(\hat{\lambda}_2,k,\overline{L}) : 0\le\tilde{x}_2, \  
\tilde{x}_1< \nu(\tilde{x}_2)
\right\} 
\label{eq:defABC_Ce}
\\
\bfD&:=\left\{ \tilde{x}\!\in\! H(\hat{\lambda}_2,k,\overline{L}) : \tilde{x}_2<0, \ 
\tilde{x}_1 \le -k\tilde{x}_2
\right\}
\label{eq:defABC_De}
\\
\bfE&:=\left\{ \tilde{x}\!\in\! H(\hat{\lambda}_2,k,\overline{L}) : \tilde{x}_2<0, \  
-k\tilde{x}_2<\tilde{x}_1 \le 
\theta^{-1}(-\tilde{x}_2)
\right\}
\label{eq:defABC_Ee}\\
\bfF&:=\left\{ \tilde{x}\!\in\! H(\hat{\lambda}_2,k,\overline{L}) : \tilde{x}_2<0, \  
\theta^{-1}(-\tilde{x}_2)<\tilde{x}_1
\right\} . 
\label{eq:defABC_Fe}
\end{align}
\end{subequations}
Clearly, we have 
\begin{align*}
\bfA\cup \bfB\cup \bfC\cup \bfD\cup \bfE\cup \bfF
=H(\hat{\lambda}_2,k,\overline{L}). 
\end{align*}
By the definition of \eqref{eq:defVe}
\eqref{eq:defVe12} and \eqref{eq:defVe3}, we have 
$\tilde{V}(\tilde{x})<\infty$ for all 
$\tilde{x}\in H(\hat{\lambda}_2,k,\overline{L})$. 
On the set of 
$\tilde{x}_{1,2}$ belonging to $H(\hat{\lambda}_2,k,\overline{L})$, 
$\tilde{V}_{1,2}(\tilde{x}_{1,2})=0$ implies 
$\tilde{x}_{1,2}=0$. Due to \eqref{eq:defVe3} and \eqref{eq:defVe},
\begin{align}
\tilde{V}(\tilde{x})=0 \ \Leftrightarrow\ \tilde{x}=0 . 
\label{eq:posidefdero}
\end{align}
By virtue of $\lambda_2-k\lambda_1>0$, the implications 
\begin{align*}
&
-k\tilde{x}_2\le\tilde{x}_1
\ \Rightarrow\  
\lambda_1\tilde{x}_1+\lambda_2\tilde{x}_2\ge 
(\lambda_2-k\lambda_1)\tilde{x}_2 
\\
&
-k\tilde{x}_2\ge\tilde{x}_1
\ \Rightarrow\  
-\lambda_1\tilde{x}_1-\lambda_2\tilde{x}_2\ge 
-(\lambda_2-k\lambda_1)\tilde{x}_2 
\end{align*}
yield 
\begin{align}
\tilde{x}\neq 0  \ \Rightarrow\ 
\tilde{V}(\tilde{x})>0.
\label{eq:posidefposi}
\end{align}
For $\tilde{x}_2\ge 0$, 
the function $\tilde{V}_{1,2}$ is continuous at $\tilde{x}_1=-k\tilde{x}_2$ , 
and $\tilde{V}_{1,2}(\tilde{x}_{1,2})=-k\lambda_1\tilde{x}_2+\lambda_2\tilde{x}_2$. 
At $\tilde{x}_1= \nu(\tilde{x}_2)$ for $\tilde{x}_2\ge 0$, 
the function $\tilde{V}_{1,2}$ is continuous 
since 
\begin{align*}
\tilde{V}_{1,2}(\tilde{x}_{1,2})
&=
P^{-1}\left(-\lambda_1\tilde{x}_1+\hat{\lambda}_2\tilde{x}_2\right) 
\\
&=
P^{-1}\left(-\lambda_1\nu(\tilde{x}_2)+\hat{\lambda}_2\tilde{x}_2\right) 
\\
&=
P^{-1}\left(-\hat{\lambda}_2\tilde{x}_2+P((\lambda_2-k\lambda_1)\tilde{x}_2)+\hat{\lambda}_2\tilde{x}_2\right) 
\\
&=
(\lambda_2-k\lambda_1)\tilde{x}_2 . 
\end{align*}
For $\tilde{x}_2<0$, 
the function $\tilde{V}_{1,2}$ is continuous at $\tilde{x}_1=-k\tilde{x}_2$, and 
$
\tilde{V}_{1,2}(\tilde{x}_{1,2})=P^{-1}((k\lambda_1-\lambda_2)\tilde{x}_2) 
$. 
At $\tilde{x}_1= \theta^{-1}(-\tilde{x}_2)$ for $\tilde{x}_2< 0$, 
the function $\tilde{V}_{1,2}$ is continuous since 
\begin{align*}
P^{-1}((k\lambda_1-\lambda_2))\tilde{x}_2)
&=
P^{-1}(\lambda_2-k\lambda_1)\theta(\tilde{x}_1))
\\
&=
\lambda_1\tilde{x}_1+
\hat{\lambda}_2\theta(\tilde{x}_1)
\\
&=
\lambda_1\tilde{x}_1-\hat{\lambda}_2\tilde{x}_2 . 
\end{align*}
For $\tilde{x}_1< 0$, 
the function $\tilde{V}_{1,2}$ is continuous at $\tilde{x}_2=0$, and 
$\tilde{V}_{1,2}(\tilde{x}_{1,2})=P^{-1}(-\lambda_1\tilde{x}_1)$. 
At $\tilde{x}_2=0$ for $\tilde{x}_1>0$, 
the function $\tilde{V}_{1,2}$ is continuous and 
$\tilde{V}_{1,2}(\tilde{x}_{1,2})=\lambda_1\tilde{x}_1$. 
These arguments verify 
\begin{align}
\overline{G}(\hat{\lambda}_2,k,\overline{L})
\subset
H(\hat{\lambda}_2,k,\overline{L}) . 
\label{eq:GincludedinH}
\end{align}

%
Define 
\begin{align*}
\overline{G}_{1,2}(\hat{\lambda}_2,k,L):= \left\{\tilde{x}\in[-\hat{x}_i,\infty)^2 : 
\tilde{V}_{1,2}(\tilde{x}_{1,2})\le L \right\}
\end{align*}
for $L\in\Rset_+$. 
Since for each $s>0$, $P(s)$ defined with $\hat{\lambda}_2=a$ 
is larger than $P(s)$ defined with $\hat{\lambda}_2=b$ for 
$0 \le a \le b$, the definition \eqref{eq:defVe12} yields
\begin{align}
0 \le a \le b\ \Rightarrow\ 
\overline{G}_{1,2}(b,k,L)\subset 
\overline{G}_{1,2}(a,k,L)
\label{eq:GLorderinhatlam212}
\end{align}
for all $L\in\Rset_+$. 
The definitions \eqref{eq:defVe} and \eqref{eq:defVe3} proves
\eqref{eq:GLorderinhatlam2} in \textbf{(ii)}. 
Property \eqref{eq:GLorderink} in \textbf{(iii)} 
is also verified from \eqref{eq:defVe12}. 

Since $\omega(s)$ is increasing in $\hat{\lambda}_2\ge 0$ for 
each $s>0$ by definition, 
$P(s)$ is decreasing in $\hat{\lambda}_2$. 
Thus the function $\nu(L)$ 
is increasing in $\hat{\lambda}_2\ge 0$ for each $L>0$. 
Hence, for each $\overline{L}\ge 0$,  
there always exists $\hat{\lambda}_2>0$  such that 
\eqref{eq:TLcorneredgeout} holds. 
In fact, for $\hat{\lambda}_2=0$ and $\lambda_0:=\lambda_2-k\lambda_1>0$ we have 
\begin{align*}
\lambda_1\theta^{-1}(-s)
-\lambda_1\nu(s)
&=
\frac{\lambda_1\hat{x}_1\lambda_0\hat{x}_2}{\lambda_0\hat{x}_2+\lambda_0s}-\lambda_1\hat{x}_1
+
\lambda_0\hat{x}_2-\frac{\lambda_1\hat{x}_1\lambda_0\hat{x}_2}{\lambda_1\hat{x}_1+\lambda_0s}
\\
&=
\frac{\lambda_0s(\lambda_0s+\lambda_1\hat{x}_1+\lambda_0\hat{x}_2)(\lambda_0\hat{x}_2-\lambda_1\hat{x}_1)}
{(\lambda_1\hat{x}_1+\lambda_0s)(\lambda_0\hat{x}_2+\lambda_0s)} 
\\
&\ge 0 
\end{align*}
for all $s\in\Rset_+$. 
The last inequality follows from \eqref{eq:lam12sire} and 
\begin{align*}
1-k=\frac{\lambda_0}{\lambda_1}\ge
\frac{\hat{x}_1}{\hat{x}_2}=
\frac{\gamma+\mu}{\mu(\hat{R}_0-1)}
\end{align*}
guaranteed by \eqref{eq:ksire}. 
Thus, property \eqref{eq:TLcorneredgeout} holds for 
all $\overline{L}\in\Rset_+$ if $\overline{\lambda}_2=0$. 
Item \textbf{(iv)} is proved. 

Continuity of the functions guarantees the existence of $\overline{\lambda}_2>0$ 
and $\overline{L}\ge 0$ satisfying 
\eqref{eq:TLcorneredgeout} for all $\hat{\lambda}_2\in(0,\overline{\lambda}_2]$. 
By virtue of \eqref{eq:defVe12} and \eqref{eq:GLorderinhatlam212}, 
for any $\tilde{x}_{1,2}\in G_{1,2}$, 
there exist $\overline{\lambda}_2>0$, $\overline{L}\ge 0$ and 
$\overline{k}\in(0,k_0)$ such that 
$\tilde{x}_{1,2}\in\overline{G}_{1,2}(\hat{\lambda}_2,k,\overline{L})$, 
\eqref{eq:TLcorneredgeslope} and \eqref{eq:TLcorneredgeout} 
are satisfied for all $\hat{\lambda}_2\in(0,\overline{\lambda}_2]$ 
and all $k\in(0,\overline{k}]$
Therefore, \eqref{eq:defVe} and \eqref{eq:defVe3} 
prove the claim \textbf{(i)}.

\subsection{Proof of Theorem \ref{thm:sireLyap}}\label{ssec:thmsireLyap}

First, recall that equation \eqref{eq:sir} is forward complete, and 
by itself guarantees 
the forward invariance of the set $[-\hat{x}_i,\infty)^3$. i.e., 
$\tilde{x}(t)\in[-\hat{x}_i,\infty)^3$ for all $t\in\Rset_+$  
with respect to all $\tilde{x}(0)\in[-\hat{x}_i,\infty)^3$ and 
$\tilde{u}(t)$ satisfying $\tilde{u}(t)\in[-\hat{B},\infty)$. 
As demonstrated in the proof of Lemma \ref{lem:sublevel}, 
for given $\hat{\lambda}_2$, $\overline{L}$, $k$ under the 
stated assumptions, 
the function $\tilde{V}$ is defined and continuous on 
$H(\hat{\lambda}_2,k,\overline{L})$, and satisfies 
\eqref{eq:posidefdero} and \eqref{eq:posidefposi}. 
We also have $\lambda_2-k\lambda_1>0$. 
Since $\theta^{-1}$ and $P^{-1}$ defined in \eqref{eq:defthetainv} and \eqref{eq:defPinv} 
are locally Lipschitz, 
the function $\tilde{V}_{1,2}$ defined by \eqref{eq:defVe12} is locally Lipschitz. 
Since $\tilde{V}_{3}$ defined by \eqref{eq:defVe3} is locally Lipschitz, 
so is $\tilde{V}$. 
Since $\lambda_2$ and $\hat{\lambda}_2$ are positive, 
the definitions \eqref{eq:defthetainv} and c imply 
$(P^{-1})^\prime(s)>0$ for all $s\in(-\infty,(\lambda_2-k\lambda_1)\hat{x}_2)$. 
In fact, 
\begin{align}
(P^{-1})^\prime(s)&=
\frac{1}{\lambda_2-k\lambda_1}\left(
\frac{\lambda_1(\lambda_2-k\lambda_1)^2\hat{x}_1\hat{x}_2}{((\lambda_2-k\lambda_1)\hat{x}_2-s)^2}
+\hat{\lambda}_2\right)
\nonumber\\
&>
\frac{\hat{\lambda}_2}{\lambda_2-k\lambda_1} . 
\label{eq:derPinv}
\end{align}
From \eqref{eq:defPinv} and the above, 
\begin{align}
&
\lim_{s\to(\lambda_2-k\lambda_1)\hat{x}_2-}P^{-1}(s)=\infty
\label{eq:Pinvlim}
\\
&
\lim_{s\to(\lambda_2-k\lambda_1)\hat{x}_2-}(P^{-1})^\prime(s)=\infty . 
\label{eq:derPinvlim}
\end{align}
Now, we evaluate the derivative of $V(x(t))$ along the solution $x(t)$ of 
\eqref{eq:sir} region by region in accordance with \eqref{eq:defABCe}.
In the region $\bfA\cap\{\tilde{x}_3\in[0,\infty)\}$, by virtue of \eqref{eq:lam12sire} and 
$f(\hat{x},\hat{B})=0$, we have 
\begin{align}\label{eq:DVA+}
\dfrac{\partial \tilde{V}}{\partial \tilde{x}}\tilde{f}
&= 
\lambda_1(B-\mu S -\beta IS)
+\lambda_2(\beta IS -\gamma I-\mu I)
+\lambda_3(\gamma I-\mu R)
\nonumber\\
&=\lambda_1(\tilde{u}
-\mu\tilde{x}_1 -(\gamma_A+\mu)\tilde{x}_2) -\lambda_3\mu\tilde{x}_3  
\nonumber\\
&\le
-\mu\tilde{V}(\tilde{x}) + \lambda_1\tilde{u} , 
\end{align}
where $\gamma_A=\gamma(1-\lambda_3/\lambda_2)$. 
In the region $\bfA\cap\{\tilde{x}_3\in[-\hat{x}_3,0]\}$, 
\begin{align}\label{eq:DVA-}
\dfrac{\partial \tilde{V}}{\partial \tilde{x}}\tilde{f}
&= 
\lambda_1(B-\mu S -\beta IS)
+\lambda_2(\beta IS -\gamma I-\mu I)
+\lambda_3(\mu R-\gamma I)
\nonumber\\
&\le\lambda_1(\tilde{u}
-\mu\tilde{x}_1 -(\gamma+\mu)\tilde{x}_2) 
+\lambda_3\mu\tilde{x}_3 
\nonumber\\
&\le
-\mu\tilde{V}(\tilde{x}) + \lambda_1\tilde{u}. 
\end{align}
In the set $\bfB\cap\{\tilde{x}_3\in[0,\infty)\}$, 
due to 
\begin{align*}
-k\tilde{x}_2>x_1-\hat{x}_1=x_1-\frac{\gamma+\mu}{\beta}, 
\end{align*}
we have 
\begin{align}\label{eq:DVB+}
\dfrac{\partial \tilde{V}}{\partial \tilde{x}}\tilde{f}
&=
(\lambda_2-k\lambda_1)(\beta IS -\gamma I-\mu I)
+\lambda_3(\gamma I-\mu R)
\nonumber\\
&\le \
(\lambda_2-k\lambda_1)\left(\gamma+\mu-k\beta\tilde{x}_2
-\gamma-\mu\right)x_2
+\lambda_3(\gamma\tilde{x}_2-\mu\tilde{x}_3)
\nonumber\\
&\le 
-k(\lambda_2-k\lambda_1)\beta\hat{x}_2\tilde{x}_2
+\lambda_3(\gamma \tilde{x}_2-\mu \tilde{x}_3)
\nonumber\\
&=
-\left(k(\lambda_2-k\lambda_1)\mu(\hat{R}_0-1)-\lambda_3\gamma\right)\tilde{x}_2-\lambda_3\mu \tilde{x}_3
\nonumber\\
&=
-\left(k\mu(\hat{R}_0-1)-\frac{\lambda_3\gamma}{\lambda_2-k\lambda_1}\right)
\tilde{V}_{1,2}(\tilde{x}_{1,2})-\lambda_3\mu \tilde{x}_3  
\nonumber\\
&\le
-a_B\tilde{V}(\tilde{x})
\end{align}
for some $a_B>0$ since 
$k\mu(\hat{R}_0-1)>{\lambda_3\gamma}/{\lambda_2-k\lambda_1}$ 
is guaranteed by \eqref{eq:lam12sire} and \eqref{eq:lam3sire}. 
In the set $\bfB\cap\{\tilde{x}_3\in[-\hat{x}_3,0]\}$, we obtain 
\begin{align}\label{eq:DVB-}
\dfrac{\partial \tilde{V}}{\partial \tilde{x}}\tilde{f}
&=
(\lambda_2-k\lambda_1)(\beta IS -\gamma I-\mu I)
+\lambda_3(\mu R-\gamma I)
\nonumber\\
&\le 
-k(\lambda_2-k\lambda_1)\beta\hat{x}_2\tilde{x}_2-\lambda_3\gamma \tilde{x}_2+\lambda_3\mu \tilde{x}_3
\nonumber\\
&=
-\left(k\lambda_2\mu(\hat{R}_0-1)+\lambda_3\gamma\right)\tilde{x}_2+\lambda_3\mu \tilde{x}_3
\nonumber\\
&=
-\left(k\mu(\hat{R}_0-1)+\frac{\lambda_3\gamma}{\lambda_2-k\lambda_1}\right)
\tilde{V}_{1,2}(\tilde{x}_{1,2})+\lambda_3\mu \tilde{x}_3 
\nonumber\\
&\le
-a_B\tilde{V}(\tilde{x}) . 
\end{align}
Since $\tilde{x}_2\le -\theta(\tilde{x}_1)$ is equivalent to 
$x_1x_2\le\hat{x}_1\hat{x}_2$, 
property \eqref{eq:TLcorneredgeout} guarantees 
$x_1x_2\le\hat{x}_1\hat{x}_2$ for all 
$\tilde{x}_{1,2}$ in $\bfC$. 
Hence, in the region $\bfC\cap\{\tilde{x}_3\in[0,\infty)\}$, we have 
\begin{align}\label{eq:DVC+}
\dfrac{\partial \tilde{V}}{\partial \tilde{x}}\tilde{f}
&= 
(P^{-1})^\prime(v)
\left(\lambda_1(\mu \tilde{x}_1-\tilde{u})
-\hat{\lambda}_2(\gamma_C+\mu)\tilde{x}_2\right)
-\lambda_3\mu \tilde{x}_3
\nonumber\\
&\hspace{15ex}
+(P^{-1})^\prime(v)(\lambda_1+\hat{\lambda}_2)\beta 
(x_1x_2-\hat{x}_1\hat{x}_2)
\nonumber\\
&\le
(P^{-1})^\prime(v)
\left(\lambda_1(\mu \tilde{x}_1-\tilde{u})
-\hat{\lambda}_2(\gamma_C+\mu)\tilde{x}_2\right)
-\lambda_3\mu \tilde{x}_3  
\nonumber\\
&\le
(P^{-1})^\prime(v)
\left(-\mu v
-\lambda_1\tilde{u}\right)
-\lambda_3\mu \tilde{x}_3  ,
\end{align}
by virtue of \eqref{eq:derPinv}, 
where $v=P(\tilde{V}_{1,2}(\tilde{x}_{1,2}))=-\lambda_1\tilde{x}_1+\hat{\lambda}_2\tilde{x}_2$. 
Note that 
\begin{align*}
\gamma_C:=\gamma\left(1-
\frac{\lambda_3(\lambda_2-k\lambda_1)}{\hat{\lambda}_2^2}
\right)\ge 0, 
\end{align*}
due to \eqref{eq:lam12sire} and \eqref{eq:lam3sire}. 
Property \eqref{eq:derPinvlim} implies the existence of 
$\alpha_{C0}\in\calK_\infty$ such that 
\begin{align}\label{eq:DVC+gas}
\tilde{u}=0 \ \Rightarrow\ 
\dfrac{\partial \tilde{V}}{\partial \tilde{x}}\tilde{f}
\le -\alpha_{C0}(\tilde{V}) . 
\end{align}
In the region $\bfC\cap\{\tilde{x}_3\in[-\hat{x}_3,0]\}$, 
\begin{align}\label{eq:DVC-}
\dfrac{\partial \tilde{V}}{\partial \tilde{x}}\tilde{f}
&= 
(P^{-1})^\prime(v)
\left(\lambda_1(\mu \tilde{x}_1-\tilde{u})
-\hat{\lambda}_2(\gamma+\mu)\tilde{x}_2\right)
+\lambda_3\mu \tilde{x}_3
\nonumber\\
&\hspace{15ex}
+(P^{-1})^\prime(v)(\lambda_1+\hat{\lambda}_2)\beta 
(x_1x_2-\hat{x}_1\hat{x}_2)
\nonumber\\
&\le
(P^{-1})^\prime(v)
\left(\lambda_1(\mu \tilde{x}_1-\tilde{u})
-\hat{\lambda}_2(\gamma+\mu)\tilde{x}_2\right)
+\lambda_3\mu \tilde{x}_3  
\nonumber\\
&\le
(P^{-1})^\prime(v)
\left(-\mu v
-\lambda_1\tilde{u}\right)
+\lambda_3\mu \tilde{x}_3  , 
\end{align}
and \eqref{eq:DVC+gas}. 
In the region $\bfD\cap\{\tilde{x}_3\in[-\hat{x}_3,0]\}$, we obtain 
\begin{align}\label{eq:DVD-}
\dfrac{\partial \tilde{V}}{\partial \tilde{x}}\tilde{f}
&= 
(P^{-1})^\prime(w)\bigl[\lambda_1(\mu S +\beta IS -B)
\nonumber\\
&\hspace{12ex}
+\lambda_2(\gamma I+\mu I -\beta IS)\bigr]
+\lambda_3(\mu R - \gamma I)
\nonumber\\
&\le
(P^{-1})^\prime(w)
\lambda_1(\mu\tilde{x}_1-\tilde{u}
+(\gamma_D+\mu)\tilde{x}_2) 
+\lambda_3\mu\tilde{x}_3 
\nonumber\\
&\le
(P^{-1})^\prime(w)
(-\mu w-\lambda_1\tilde{u}) 
+\lambda_3\mu\tilde{x}_3  ,
\end{align}
by virtue of \eqref{eq:derPinv}, 
where $w=P(\tilde{V}_{1,2}(\tilde{x}_{1,2}))=-\lambda_1\tilde{x}_1-\lambda_2\tilde{x}_2$ and 
\begin{align*}
\gamma_D:=\gamma\left(1-
\frac{\lambda_3(\lambda_2-k\lambda_1)}{\lambda_2\hat{\lambda}_2}
\right)\ge 0, 
\end{align*}
Here, $\gamma_D \ge 0$ follows from 
$k\in(0,1)$, \eqref{eq:lam12sire} and \eqref{eq:lam3sire}. 
The existence of $\alpha_{D0}\in\calK_\infty$ such that 
\begin{align}\label{eq:DVD-gas}
\tilde{u}=0 \ \Rightarrow\ 
\dfrac{\partial \tilde{V}}{\partial \tilde{x}}\tilde{f}
\le -\alpha_{D0}(\tilde{V}) 
\end{align}
also follows from \eqref{eq:derPinvlim}. 
In the case of $\tilde{x}\in \bfD\cap\{\tilde{x}_3\in[0,\infty)\}$, we have 
\begin{align}\label{eq:DVD+}
\dfrac{\partial \tilde{V}}{\partial \tilde{x}}\tilde{f}
&= 
(P^{-1})^\prime(w)\bigl[\lambda_1(\mu S +\beta IS -B)
\nonumber\\
&\hspace{12ex}
+\lambda_2(\gamma I+\mu I -\beta IS)\bigr]
+\lambda_3(\gamma I - \mu R)
\nonumber\\
&\le
(P^{-1})^\prime(w)
\lambda_1(\mu\tilde{x}_1-\tilde{u}
+(\gamma+\mu)\tilde{x}_2) 
-\lambda_3\mu\tilde{x}_3 
\nonumber\\
&\le
(P^{-1})^\prime(w)
(-\mu w-\lambda_1\tilde{u}) 
-\lambda_3\mu\tilde{x}_3  ,  
\end{align}
and \eqref{eq:DVD-gas}. 
In the case of  $\tilde{x}\in \bfE\cap\{\tilde{x}_3\in[-\hat{x}_3,0]\}$, from 
\begin{align*}
-k\tilde{x}_2<x_1-\hat{x}_1=x_1-\frac{\gamma+\mu}{\beta}
\end{align*}
and \eqref{eq:derPinvlim} we obtain 
\begin{align}\label{eq:DVE-}
\dfrac{\partial \tilde{V}}{\partial \tilde{x}}\tilde{f}
&=
(P^{-1})^\prime(z)
(\lambda_2-k\lambda_1)(\gamma I+\mu I-\beta IS )
+\lambda_3(\mu R-\gamma I)
\nonumber\\
&\le 
(P^{-1})^\prime(z)
(\lambda_2-k\lambda_1)\left(\gamma+\mu +k\beta\tilde{x}_2
-\gamma-\mu\right)x_2
+\lambda_3(\mu\tilde{x}_3-\gamma\tilde{x}_2)
\nonumber\\
&\le 
(P^{-1})^\prime(z)
(\lambda_2-k\lambda_1)
(\hat{x}_2-\theta\circ\omega^{-1}(\overline{L}))
\beta\tilde{x}_2
-\lambda_3\gamma \tilde{x}_2+\lambda_3\mu \tilde{x}_3
\nonumber\\
&\le 
-(P^{-1})^\prime(z)(\hat{x}_2-\theta\circ\omega^{-1}(\overline{L}))
\beta\gamma_E z
+\lambda_3\mu \tilde{x}_3  
\nonumber\\
&\le 
-\alpha_E(\tilde{V}(\tilde{x}))
\end{align}
for some $\alpha_E\in\calK_\infty$, 
where $z=P(\tilde{V}_{1,2}(\tilde{x}_{1,2}))=(k\lambda_1-\lambda_2)\tilde{x}_2$ and 
\begin{align*}
\gamma_E:=1-
\frac{\lambda_3\gamma}{\hat{\lambda}_2
(\hat{x}_2-\theta\circ\omega^{-1}(\overline{L}))\beta}
> 0.
\end{align*}
Here, \eqref{eq:lam12sire} and \eqref{eq:lam3sire} imply 
the above inequality. 
In the case of  $\tilde{x}\in \bfE\cap\{\tilde{x}_3\in[0,\infty)\}$, we have 
\begin{align}\label{eq:DVE+}
\dfrac{\partial \tilde{V}}{\partial \tilde{x}}\tilde{f}
&=
(P^{-1})^\prime(z)
(\lambda_2-k\lambda_1)(\gamma I+\mu I-\beta IS )
+\lambda_3(\gamma I-\mu R)
\nonumber\\
&\le 
(P^{-1})^\prime(z)
(\lambda_2-k\lambda_1)
(\hat{x}_2-\theta\circ\omega^{-1}(\overline{L}))
\beta\tilde{x}_2
+\lambda_3\gamma \tilde{x}_2
-\lambda_3\mu \tilde{x}_3
\nonumber\\
&\le 
-(P^{-1})^\prime(z)(\hat{x}_2-\theta\circ\omega^{-1}(\overline{L}))
\beta z
-\lambda_3\mu \tilde{x}_3  
\nonumber\\
&\le 
-\alpha_E(\tilde{V}(\tilde{x})) .
\end{align}
In the region $\bfF\cap\{\tilde{x}_3\in[-\hat{x}_3,0]\}$, 
since $\theta^{-1}(-\tilde{x}_2)<\tilde{x}_1$ is equivalent to 
$x_1x_2>\hat{x}_1\hat{x}_2$, we obtain 
\begin{align}\label{eq:DVF-}
\dfrac{\partial \tilde{V}}{\partial \tilde{x}}\tilde{f}
&= 
\lambda_1(\tilde{u}-\mu \tilde{x}_1)
+\hat{\lambda}_2(\gamma_F+\mu)\tilde{x}_2
+\lambda_3\mu \tilde{x}_3
-(\lambda_1+\hat{\lambda}_2)\beta 
(x_1x_2-\hat{x}_1\hat{x}_2)
\nonumber\\
&\le
\lambda_1(\tilde{u}-\mu \tilde{x}_1)
+\hat{\lambda}_2(\gamma_F+\mu)\tilde{x}_2
+\lambda_3\mu \tilde{x}_3
\nonumber\\
&\le
-\mu \tilde{V}(\tilde{x})
+\lambda_1\tilde{u}.
\end{align}
where 
$\gamma_F:=\gamma(1-{\lambda_3}/{\hat{\lambda}_2})\ge 0$ 
is implied by \eqref{eq:lam12sire} and \eqref{eq:lam3sire}. 
In the case of $\bfF\cap\{\tilde{x}_3\in[0,\infty)\}$, we have 
\begin{align}\label{eq:DVF+}
\dfrac{\partial \tilde{V}}{\partial \tilde{x}}\tilde{f}
&\le  
\lambda_1(\tilde{u}-\mu \tilde{x}_1)
+\hat{\lambda}_2(\gamma+\mu)\tilde{x}_2
-\lambda_3\mu \tilde{x}_3
-(\lambda_1+\hat{\lambda}_2)\beta 
(x_1x_2-\hat{x}_1\hat{x}_2)
\nonumber\\
&\le
-\mu \tilde{V}(\tilde{x})
+\lambda_1\tilde{u}
\end{align}
%

Therefore, 
since \eqref{eq:DVA+}, \eqref{eq:DVA-}, 
\eqref{eq:DVB+}, \eqref{eq:DVB-}, 
\eqref{eq:DVC+gas}, 
\eqref{eq:DVD-gas}, 
\eqref{eq:DVE-}, \eqref{eq:DVE+}, 
\eqref{eq:DVF-} and \eqref{eq:DVF+} 
cover ${\partial \tilde{V}}/{\partial \tilde{x}\cdot}\tilde{f}$ on 
the entire $H(\hat{\lambda}_2,k,\overline{L})$ 
except on the boundaries between the regions, 
as in the argument in the proof of Theorem \ref{thm:sirf}, 
the equilibrium $\tilde{x}=0$ is 
asymptotically stable for $\tilde{u}=0$, 
The inclusion \eqref{eq:GincludedinH} and the forward invariance of 
$[-\hat{x}_i,\infty)^3$ imply that 
the set 
$\overline{G}(\hat{\lambda}_2,k,\overline{L})$ is forward invariant 
and belongs to the domain of attraction for $\tilde{u}=0$.

Next, define
\begin{align*}
&
Q(\overline{L})=\left\{[\tilde{V}_{1,2},\tilde{V}_3]^T\in\Rset_+^2 : 
\exists L\le [\overline{L},\infty)\hspace{1.5ex} 
\tilde{V}_{1,2}+\tilde{V}_3=L\right\}
\\
&
\eta(\overline{L})=\min_{[\tilde{V}_{1,2},\tilde{V}_3]^T\in Q(\overline{L})}
P(\tilde{V}_{1,2})
+\frac{\lambda_3\tilde{V}_3}{(P^{-1})^\prime(P(\tilde{V}_{1,2}))}
\end{align*}
for $\overline{L}\in\Rset_+$. 
By definition, $\eta$ is of class $\calP$  and non-decreasing. Furthermore, 
the definition \eqref{eq:defP} gives 
\begin{align}
\lim_{s\to\infty}\eta(s)=\lim_{s\to\infty}P(s)
=(\lambda_2-k\lambda_1)\hat{x}_2
\end{align}
since $\tilde{V}_{1,2}< \tilde{V}_{1,2}+\tilde{V}_3=\infty$ 
implies $\tilde{V}_3=\infty$ and 
${\lambda_3\tilde{V}_3/(P^{-1})^\prime(P(\tilde{V}_{1,2}))}=\infty$. 
From \eqref{eq:DVC+} and \eqref{eq:DVC-}, in region $\bfC$, 
\begin{align}\label{eq:DVCiss}
&\tilde{u}\in[-\delta\mu\eta(\tilde{V})/\lambda_1,\infty)
\ \Rightarrow\ 
\nonumber\\
& \hspace{1.3ex}
\dfrac{\partial \tilde{V}}{\partial \tilde{x}}\tilde{f}
\le -(1-\delta)(
(P^{-1})^\prime(v)\mu v+\lambda_3\mu |\tilde{x}_3|) 
\le -\alpha_C(\tilde{V})  
\end{align}
holds for some $\alpha_C\in\calK_\infty$. 
Applying the same argument to \eqref{eq:DVD-} and \eqref{eq:DVD+} leads to 
\begin{align}\label{eq:DVDiss}
&\tilde{u}\in[-\delta\mu\eta(\tilde{V})/\lambda_1,\infty)
\ \Rightarrow\ 
\nonumber\\
& \hspace{1.3ex}
\dfrac{\partial \tilde{V}}{\partial \tilde{x}}\tilde{f}
\le -(1-\delta)(
(P^{-1})^\prime(w)\mu w+\lambda_3\mu |\tilde{x}_3|) 
\le -\alpha_D(\tilde{V})  
\end{align}
with $\alpha_D=\alpha_C$ for region $\bfD$. 
On the other hand, due to \eqref{eq:DVA+}, \eqref{eq:DVA-}, 
\eqref{eq:DVF-} and \eqref{eq:DVF+}, in $\bfA$ and $\bfF$ we have 
\begin{align}\label{eq:DVAFiss}
\tilde{u}\in(-\infty, \delta\mu\tilde{V}/\lambda_1]
\ \Rightarrow\ 
\dfrac{\partial \tilde{V}}{\partial \tilde{x}}\tilde{f}
\le -(1-\delta)\mu\tilde{V} . 
\end{align}
Let $\zeta: \Rset\to(-\delta\mu P(\overline{L})/\lambda_1,\delta\mu\overline{L}/\lambda_1)$ 
be a bijective continuous function satisfying $\zeta(0)=0$. 
Define $r=\zeta^{-1}(\tilde{u})$. 
Properties \eqref{eq:DVAFiss}, 
\eqref{eq:DVB+}, \eqref{eq:DVB-}, 
\eqref{eq:DVCiss}, 
\eqref{eq:DVDiss}, 
\eqref{eq:DVE-} and \eqref{eq:DVE+} imply the existence of $\chi\in\calK$ 
and $\alpha\in\calK_\infty$ such that 
\begin{subequations}\label{eq:DVABCDEF}
\begin{align}
&
\tilde{V}(\tilde{x}))\ge \chi(|r|) 
\ \Rightarrow\ 
\dfrac{\partial \tilde{V}}{\partial \tilde{x}}\tilde{f}\le
-\alpha(\tilde{V}(\tilde{x}))
\label{eq:DVABCDEFderV}
\\
&
\overline{L}\ge \chi(|r|)
\label{eq:DVABCDEFderChi}
\end{align}
\end{subequations}
are satisfied for all $\tilde{x}\in H(\hat{\lambda}_2,k,\overline{L})$ and all 
$r(t)\in\Rset$ with any given $\delta\in(0,1)$ 
except on the boundaries between the regions. 
Here, 
$\tilde{u}\in
(-\delta\mu P(\overline{L})/\lambda_1,\delta\mu\overline{L}/\lambda_1)$ 
guarantees the achievement of \eqref{eq:DVABCDEFderChi}. 
With the help of the forward invariance of $[-\hat{x}_i,\infty)^3$, 
property \eqref{eq:DVABCDEF} implies that 
$\tilde{x}(0)\in\overline{G}(\hat{\lambda}_2,k,\overline{L})$
yields $\tilde{x}(t)\in\overline{G}(\hat{\lambda}_2,k,\overline{L})$ 
for all $t\in\Rset_+$ as long as $\tilde{u}$ satisfies \eqref{eq:issuL}. 
Invoking the argument of the lower Dini derivative again, 
property \eqref{eq:DVABCDEF} also imply 
ISS of the SIR model \eqref{eq:sir} with respect to 
the input $\tilde{u}$ satisfying \eqref{eq:issuL} \cite{SONISSV}. 
In fact, the function $\tilde{V}(\tilde{x})$ defined in \eqref{eq:defVe} is an ISS 
Lyapunov function on the compact set $\overline{G}(\hat{\lambda}_2,k,\overline{L})$ 
for the given $\hat{\lambda}_2$, $\overline{L}$, $k>0$.

\subsection{Proof of Theorem \ref{thm:sire}}\label{ssec:thmsire}

Define 
\begin{align*}
&
T:=\left\{ \tilde{x}\in\Omega : 
\tilde{x}_1\le -k\tilde{x}_2, \ 
\tilde{x}_2\le 0
\right\}
\\
&
W(\tilde{x}):=-\tilde{x}_1-\tilde{x}_2+|\tilde{x}_3| .
\end{align*}
When $\tilde{x}\in T$, $\tilde{x}_3< 0$ and $\tilde{u}=0$ hold, 
the function $W(\tilde{x})$ satisfies 
\begin{align*}
\dfrac{\partial \tilde{W}}{\partial \tilde{x}}\tilde{f}
&= 
\mu S +\beta IS -B
+\gamma I+\mu I -\beta IS
+\mu R - \gamma I
\nonumber\\
&=
\mu\tilde{x}_1
+(\gamma-\gamma+\mu)\tilde{x}_2 
+\mu\tilde{x}_3
\nonumber\\
&=
-\mu W(\tilde{x})
\end{align*}
When $\tilde{x}\in T$, $\tilde{x}_3\ge 0$
and $\tilde{u}=0$ hold, we have 
\begin{align*}
\dfrac{\partial \tilde{W}}{\partial \tilde{x}}\tilde{f}
&= 
\mu S +\beta IS -B
+\gamma I+\mu I -\beta IS
+\gamma I - \mu R
\nonumber\\
&=
\mu\tilde{x}_1
+(2\gamma+\mu)\tilde{x}_2 
-\mu\tilde{x}_3 
\nonumber\\
&\le
-\mu W(\tilde{x})
\end{align*}
By virtue of \eqref{eq:Nest} with $\overline{B}=\hat{B}$, 
Lemma \ref{lem:sublevel} and the forward invariance of the set 
$[-\hat{x}_i,\infty)^3$, 
for each $x(0)\in T$, 
there exists $t_T\in[0,\infty)$, 
$\hat{\lambda}_2$, $\overline{L}$, $k> 0$ such that 
$x(t_T)\in\overline{G}(\hat{\lambda}_2,k,\overline{L})$. 
Therefore, Theorem \ref{thm:sireLyap} with $\tilde{u}=0$
shows that 
any compact set in $\Omega$ is contained in the domain of attraction. 


Next, 
writing $\overline{G}(\hat{\lambda}_2,k,\overline{L})$ as $\overline{G}$, 
Lemma \ref{lem:sublevel} guarantees that 
for any given compact set $\underline{G}$ contained in the interior of $G$, 
there exist sufficiently small $\hat{\lambda}_2$, $1/\overline{L}$, $k> 0$ such that 
$\overline{G}\supset\underline{G}$ is satisfied. 
As proved in Theorem \ref{thm:sireLyap}, 
there exist $\Psi_G\in\calKL$ and $\Gamma_G\in\calK$ such that 
\begin{align}
&
\forall t\in\Rset_+\hspace{1.5ex}
|\tilde{x}(t)|\le
\Phi_G(|\tilde{x}(0)|, t) + \Gamma_G({\esssup}_{t\in[0,t_T)}|\tilde{u}(t)|) 
\label{eq:ISSinG}
\\
&
\tilde{x}(t)\in\overline{G}(\hat{\lambda}_2,k,\overline{L})
\label{eq:invariancedist}
\end{align}
are satisfied for all $\tilde{x}(0)\in\overline{G}(\hat{\lambda}_2,k,\overline{L})$ 
and \eqref{eq:issuL}. 
Choose $|\cdot|$ as $1$-norm for consistency. 
Recall that \eqref{eq:Nest} holds 
for all $x(0)\in\Rset_+^3$ and 
all $B(t)\in[0,\overline{B}]$ with respect to an arbitrarily given constant 
$\overline{B}\ge 0$. 
Pick any $\Phi\in\calKL$ and $\Gamma\in\calK$ satisfying 
\begin{align}
&
\forall t\in\Rset_+\hspace{1.5ex}
\forall s\in\Rset_+\hspace{1.5ex}
\Phi(s,t)\ge \max\left\{\Phi_G(s,t),\, (s+|\hat{x}|)e^{-t}\right\}
\\
&
\forall t\in\Rset_+\hspace{1.5ex}
\forall s\in[0, \overline{u})\hspace{1.5ex}
\Gamma(s)\ge \min\left\{\Gamma_G(s),\, s+\hat{u}+|\hat{x}|\right\}
\label{eq:siregainoinner}
\\
&
\forall t\in\Rset_+\hspace{1.5ex}
\forall s\in[\overline{u},\infty)\hspace{1.5ex}
\Gamma(s)\ge s+\hat{u}+|\hat{x}| , 
\label{eq:siregainouter}
\end{align}
where 
$\overline{u}:=\min\{\delta\mu{P(\overline{L})}/{\lambda_1},
\delta\mu\overline{L}/\lambda_1\}$. 
Using $|x|\le |\tilde{x}|+|\hat{x}|$ and $|\tilde{x}|\le |x|+|\hat{x}|$ 
one arrives at 
\begin{align*}
\forall t\in\Rset_+\hspace{1.5ex}
|\tilde{x}(t)|\le
\Phi(|\tilde{x}(0)|, t) + \Gamma({\esssup}_{t\in\Rset_+}|\tilde{u}(t)|)  
\end{align*}
for all $\tilde{x}(0)\in\overline{G}(\hat{\lambda}_2,k,\overline{L})$ and 
all $\tilde{u}(t)\in[-\hat{B},\infty)$. 

\begin{remark}
As seen in \eqref{eq:ksire} and \eqref{eq:lam3sire}, 
the parameters $k$ and  $\lambda_3$
approach zero as $\overline{L}$ tends to $\infty$. 
Hence, the sublevel sets are expanded significantly in the $x_3$-direction. 
It allows the recovered population to increase, which is 
not bad in the control of infectious diseases. 
However, it is only an upper bound, and 
the recovered population does not necessarily swell that much. 
Indeed, we have the estimate \eqref{eq:Nest}. 
\end{remark}

\begin{remark}\label{rem:linISSgain}
For large magnitude of the input $\tilde{u}$, 
an ISS-gain function obtained in the proof of Theorem \ref{thm:sire} is 
bounded from above by a linear function as in \eqref{eq:siregainouter}. 
A linear bound of the ISS-gain function $\Gamma$ can also be verified 
for small magnitude of $\tilde{u}$ in \eqref{eq:siregainoinner}. 
In fact, the property $0<(P^{-1})^\prime(0)<\infty$ obtained from 
\eqref{eq:derPinv} implies that $\eta^{-1}$ can be bounded from 
above by a linear function in a neighborhood of the origin. 
Combining \eqref{eq:DVAFiss}, 
\eqref{eq:DVB+}, \eqref{eq:DVB-}, 
\eqref{eq:DVCiss}, 
\eqref{eq:DVDiss}, 
\eqref{eq:DVE-} and \eqref{eq:DVE+} 
leads to \eqref{eq:ISSinG} with a function $\Gamma_G$ which is 
bounded from above by a linear 
function in a neighborhood of the origin. 
Thus, a linear bound of $\Gamma$ in a neighborhood of the origin
follows from \eqref{eq:siregainoinner}. 
Therefore, for all magnitude of the input $\tilde{u}$, 
the ISS-gain function of the SIR model \eqref{eq:sir} is bounded from above by 
a linear function. 
\end{remark}

\section{Difficulties and Keys for Lyapunov Construction}\label{sec:keys}

The Lyapunov functions \eqref{eq:defVf} and \eqref{eq:defVe} proposed in 
this paper depict geometric structure with slopes and regions 
which the SIR model \eqref{eq:sir} requires. 
Note that the switching with sharp edges causing non-differentiability 
is not essential, but for simply highlighting the geometrical structure 
of sublevel sets. 
In fact, if one admits complexity sacrificing explicit analytical 
expression, numerical computation can help smooth out the edges to obtain 
differentiable Lyapunov functions.  
This section explains some of major components of the geometric structure, and 
elucidates points having hampered previous studies, and how this 
paper addresses those points to estimate reasonable domains of attraction 
without resorting to LaSalle's invariance principle. 
In the previous sections, 
all the derivatives of the constructed Lyapunov functions along trajectories 
${\partial \tilde{V}}/{\partial \tilde{x}}\cdot\tilde{f}$ 
are negative except at the target equilibrium   
in the absence of perturbation $\tilde{u}$.  
Such functions are referred to strict Lyapunov functions 
in the field of control \cite{MalFredStrLyapBook09}. 
The strict negativity has allowed us to prove ISS of the SIR model  
n the presence of the perturbation.

Everyone notices the conservation of populations taking 
place in between \eqref{eq:sirS} and \eqref{eq:sirI} through $\beta IS$. 
In the two regions 
\begin{align}
&
\hat{\bfB}:=\left\{ x\in\Rset_*^3 : 
x_1<\hat{x}_1 ,\ \hat{x}_1\hat{x}_2< x_1x_2
\right\}
\label{eq:blininearregionB}
\\
&
\hat{\bfE}:=\left\{ x\in\Rset_+^3 : 
x_1>\hat{x}_1 ,\ \hat{x}_1\hat{x}_2> x_1x_2
\right\} , 
\label{eq:blininearregionE}
\end{align}
the bilinear term $\beta IS$ in \eqref{eq:sirS} generates force to let $x_1$ stay 
away from the equilibrium $\hat{x}_1$ of interest. 
Hence, in $\hat{\bfB}$ and $\hat{\bfE}$, $x_2$ and $x_3$ should dominate 
the Lyapunov function in making its derivative negative. 
In the case $\hat{R}_0<1$ of the disease-free equilibrium, 
region $\hat{\bfE}$ disappears since $\hat{x}_2=0$. 
This structure of $\hat{\bfB}$ and $\hat{\bfE}$ is incorporated in 
the definition of $\tilde{V}$ and the partitioning functions in 
\eqref{eq:defVf} and \eqref{eq:defVe12}. 
To define a set taking care of $\hat{\bfB}$, 
the disease-free case can use a linear function in \eqref{eq:defVf} 
since $\tilde{x}_2$ is non-negative as discussed at the beginning 
of Section \ref{sec:lyapxe}.

Functions in the form of 
\begin{align}
\tilde{V}(\tilde{x})=
\tilde{V}_1(\tilde{x}_1)+\tilde{V}_2(\tilde{x}_2)+\tilde{V}_3(\tilde{x}_3)
\label{eq:Vsum}
\end{align}
have been widely used as Lyapunov functions in stability analysis and 
design of dynamical systems. They are often referred to as 
sum-separable (Lyapunov) functions or scalar (Lyapunov) functions 
\cite{DIRR15,LSSurvey83}. 
In this paper, let a function $\tilde{V}(\tilde{x}): \Rset_+^3\to\Rset_+$ be said to 
be separable if 
\begin{align}
j\ne i \ \Rightarrow \ 
\forall \tilde{x}\hspace{1.5ex}
\frac{\partial^2 \tilde{V}}{\partial \tilde{x}_j \partial \tilde{x}_i}=0. 
\label{eq:separability}
\end{align}
Clearly, continously differentiable functions in the form of 
\eqref{eq:Vsum} are separable\footnote{
The max-separable functions which are also popular in the literature \cite{JIA96LYA,LSSurvey83,DASITOWIRejc11,DIRR15}
are not separable in the sense of \eqref{eq:separability} since 
the switching depends on the whole $\tilde{x}$ instead of the individual $\tilde{x}_i$.}. 
The structure \eqref{eq:separability} is very popular and useful for constructing a Lyapunov 
function since the negativity of its derivative can be assessed by looking at 
components separately as 
\begin{align}
\frac{\partial \tilde{V}}{\partial \tilde{x}}(\tilde{x})
\tilde{f}(\tilde{x},\tilde{u})
=\sum_{i=1}^3
\frac{\partial \tilde{V}}{\partial \tilde{x}_i}(\tilde{x}_i)
\tilde{f}_i(\tilde{x},\tilde{u}) . 
\end{align}
and focusing on the interaction between subsystems 
$\tilde{x}_i=\tilde{f}_i(\tilde{x},\tilde{u})$, $i=1,2,3$ (see \cite{HILLMOY2,ITOTAC06,DASITOWIRejc11,MIRITOparasg15} and references therein). 
In fact, for popular models of infectious diseases, 
many preceding studies use the sum-separable form \eqref{eq:Vsum} 
(e.g., \cite{KOROLyap02,KOROLyap04,KOROgennonID06,OREGAN2010446,EnaNakIDlyapdelay11,SHUAIIDlyapu13,OREGAN2010446,BichSIRLyap14,FALLIDlypu07}).

There is a major difference between the endemic equilibrium and 
the disease-free equilibrium in constructing a Lyapunov function. 
The endemic case exhibits spiral trajectories around the equilibrium 
on the $S$-$I$ plane, i.e., the origin $\tilde{x}_{1,2}=0$ of the 
$(\tilde{x}_1,\tilde{x}_2)$-plane. 
If 
\begin{align}
x_1=\hat{x}_1=\frac{\gamma+\mu}{\beta}, 
\label{eq:undesirablestate}
\end{align}
then the SIR model \eqref{eq:sir} gives 
$\dot{x}_2=0$ and 
\begin{subequations}\label{eq:undesirablestatemov}
\begin{align}
\hat{x}_2<x_2 & \Rightarrow\ \dot{x}_1<0
\\
\hat{x}_2>x_2 & \Rightarrow\ \dot{x}_1>0 . 
\end{align}
\end{subequations}
No matter how far and close $x_2$ is to $\hat{x}_2$, 
this anti-parallel structure \eqref{eq:undesirablestatemov} of flows takes place. 
It disappears only at the equilibrium $x_2=\hat{x}_2$. 
Since $\dot{x}_2=\dot{\tilde{x}}_2=0$ hold for \eqref{eq:undesirablestate}, 
a function $\tilde{V}(\tilde{x})$ of the form \eqref{eq:separability} 
exhibits the decrease ${\partial \tilde{V}}/{\partial \tilde{x}}\cdot\tilde{f}<0$ 
for $\tilde{x}_2 \ne 0$ (i.e., $x_2 \ne \hat{x}_2$) only if 
\begin{subequations}\label{eq:undescontra}
\begin{align}
\hat{x}_3<
x_3\le \frac{\gamma}{\mu}x_2, \ 
\hat{x}_2<x_2 & \Rightarrow\ 
\left.\dfrac{\partial \tilde{V}}{\partial \tilde{x}_1}(\tilde{x}_1)\right|_{\tilde{x}_1=0}>0
\label{eq:undescontraup}
\\
\hat{x}_3>x_3\ge 
\frac{\gamma}{\mu} x_2, \ 
\hat{x}_2>x_2 & \Rightarrow\ 
\left.\dfrac{\partial \tilde{V}}{\partial \tilde{x}_1}(\tilde{x}_1)\right|_{\tilde{x}_1=0}<0 ,  
\label{eq:undescontradown}
\end{align}
\end{subequations}
provided that 
\begin{subequations}\label{eq:sepabowl}
\begin{align}
\hat{x}_3<x_3 & \Rightarrow\ 
\dfrac{\partial \tilde{V}}{\partial \tilde{x}_3}(\tilde{x}_3)\ge 0
\\
\hat{x}_3>x_3 & \Rightarrow\ 
\dfrac{\partial \tilde{V}}{\partial \tilde{x}_3}(\tilde{x}_3)\le 0 . 
\end{align}
\end{subequations}
The two conclusions in \eqref{eq:undescontra} contradict each other. 
This situation is illustrated by Fig. \ref{fig:keys} (a) on 
$(x_1,x_2)$-plane. 
The positive definiteness of $\tilde{V}$ requires \eqref{eq:sepabowl}
at least locally at $\tilde{x}_3=0$, i.e., in a neighborhood of 
$\tilde{x}_3=0$. Thus, any (piecewise) continuously differentiable function 
$\tilde{V}(\tilde{x})$ which is separable \eqref{eq:separability} 
cannot be a Lyapunov function in the sense of 
${\partial \tilde{V}}/{\partial \tilde{x}}\cdot\tilde{f}<0$. 
It is worth mentioning that property \eqref{eq:sepabowl} is usually 
employed in the region of interest, instead of 
the existence of a small neighborhood of $\tilde{x}_3=0$. 
In obtaining reasonable level sets to 
secure an estimate of domain of attraction, 
violating \eqref{eq:sepabowl} is usually too hard. 
The Lyapunov function $\tilde{V}(\tilde{x})$ constructed in 
\eqref{eq:defVe} is not separable. 
In fact, the conditions of the partitioning in \eqref{eq:defVe12} 
require both $x_1$ and $x_2$. 
Importantly, 
the second case \eqref{eq:undescontradown} disappears 
from \eqref{eq:undescontra} in the disease-free case since $
\hat{x}_2=\hat{x}_3=0$. Thus, 
the contradiction does not rise in the disease-free case. 
This is why \eqref{eq:defVe12} employed the slope $k>0$, while 
\eqref{eq:defVf} does not.


As seen in the definition \eqref{eq:defG} of $G$, 
Theorem \ref{thm:sire} dealing with the endemic equilibrium $x_f$ 
does not cover a triangle region at the corner of 
$x_1$-axis and $x_2$-axis. 
No matter how much one modifies Lyapunov functions, there remains 
an uncovered region of non-zero volume at that corner along the $x_1$-axis. 
To see this, notice that \eqref{eq:sirS} and 
\eqref{eq:sirI} satisfy the implication 
\begin{align}
x_1<\hat{x}_1, \ x_2>0
& \Rightarrow\ \dot{x}_2<0  
\label{eq:I2equilidown}
\\
x_2=0 , \ x_1<x_{f,1} & \Rightarrow\ \dot{x}_2=0 , \ \dot{x}_1>0. 
\label{eq:I2equilizero}
\end{align}
Here, $\hat{x}_1=(\gamma+\mu)/\beta$ and $x_{f,1}=B/\mu$. The relationship 
$\hat{x}_1<x_{f,1}$ follows from $\hat{R}_0>1$. 
Define 
\begin{align}
\hat{\bfD}:=\left\{
\tilde{x}\in\Rset_+^3 : \tilde{x}_1< 0, \ -\hat{x}_2<\tilde{x}_2< 0
\right\}. 
\label{eq:defDhat}
\end{align}
Consider an initial state $x(0)\in \hat{\bfD}$ which is arbitrarily 
close to a point $[x_1(0),0_,0]^T$ for some $x_1(0)\in(0,\hat{x}_1)$. 
According to \eqref{eq:I2equilidown} and \eqref{eq:I2equilizero}, 
the trajectory $x(t)$ flows along 
the plane of $x_2=0$ ($x_1$-axis on $(x_1,x_2)$-plane) by 
decreasing its distance to the plane ($x_1$-axis) further. 
The level set of a Lyapunov function 
passing through the point $x=x(0)$ must be intersected transversally 
by the trajectory $x(t)$ inward. Hence, the level set must intersect 
the plane (the $x_1$-axis). 
Due to \eqref{eq:I2equilizero}, 
that level set crossing over\footnote{
Since $\{x\in\Rset_+^3\}$ is forward invariant for \eqref{eq:sir}, 
one can consider any artificial flow for $x\not\in\Rset_+^3$}  
$x_1$-axis on $(x_1,x_2)$-plane can never 
cross $x_1$-axis again as long as $x_1<x_{f,1}$. 
This implies the existence of a sublevel set to which the 
equilibrium $x_f$ belongs. 
At the non-target equilibrium $x_f$,  
the derivative of any Lyapunov function candidate $\tilde{V}$ along the 
trajectory is zero. Hence, the function 
$\tilde{V}$ is not a strict Lyapunov function for the target equilibrium $x_e$. 
This mechanism is illustrated in Fig. \ref{fig:keys} (b). 
In this way, independently of methods of constructing a Lyapunov function, 
there is an area remaining uncovered by 
any sublevel sets along $x_1$-axis in region $\hat{D}$. 
Theorem \ref{thm:sire} achieves the construction of a Lyapunov function 
by avoiding that prohibited region intentionally. 

\begin{figure}[t]
\unitlength=0.44mm
\begin{center}
\begin{tabular}{cc}
\begin{picture}(88,94)(3,-10)
\put(10,0){\vector(0,1){70}}
\put(10,0){\vector(1,0){80}}
\put(5,-4){\makebox(0,0){$0$}}
\put(97,0){\makebox(0,0){$x_1$}}
\put(10,76){\makebox(0,0){$x_2$}}
\multiput(40,0)(0,1.2){58}{\line(0,1){0.6}}
\multiput(10,24)(1.2,0){66}{\line(1,0){0.6}}
\put(40,-6){\makebox(0,0){$\hat{x}_1$}}
\put(3,22){\makebox(0,0){$\hat{x}_2$}}
\put(40,24){\circle*{3}}
\put(54,29){\makebox(0,0){$\hat{x}\!=\!x_e$}}
\thicklines
\color{magenta}
\put(40,43){\line(0,1){12}}
\put(40,3){\line(0,1){12}}
\color{red}
\put(46,43){\line(0,1){12}}
\put(34,3){\line(0,1){12}}
\put(49,61){\makebox(0,0){{\small high}}}
\put(31,18){\makebox(0,0){{\small high}}}
\color{blue}
\put(34,43){\line(0,1){12}}
\put(46,3){\line(0,1){12}}
\put(31,61){\makebox(0,0){{\small low}}}
\put(49,18){\makebox(0,0){{\small low}}}
\color{black}
\thicklines
\put(51,49){\vector(-1,0){24}}
\put(29,9){\vector(1,0){24}}
\end{picture}
&
\begin{picture}(98,94)(-2,-10)
\put(7,35){\vector(0,1){35}}
\put(7,35){\vector(1,0){83}}
\put(2,31){\makebox(0,0){$0$}}
\put(93,30){\makebox(0,0){$x_1$}}
\put(7,76){\makebox(0,0){$x_2$}}
\multiput(41,0)(0,1.2){61}{\line(0,1){0.6}}
\multiput(7,65)(1.2,0){68}{\line(1,0){0.6}}
\put(41,-6){\makebox(0,0){$\hat{x}_1$}}
\put(0,65){\makebox(0,0){$\hat{x}_2$}}
\put(65,35){\circle*{3}}
\put(41,65){\circle*{3}}
\put(55,70){\makebox(0,0){$\hat{x}\!=\!x_e$}}

\put(67,39.5){\makebox(0,0){$x_f$}}
\put(27,18){\makebox(0,0){
$\underbrace{\rule{1ex}{0ex}}_{\mbox{\small required\rule{4ex}{0ex}}}$
}}
\put(68,17){\makebox(0,0){
$\underbrace{\rule{1ex}{0ex}}_{\small \mbox{\rule{4ex}{0ex}possible}}$
}}
\put(54,52){\makebox(0,0){
$\overbrace{\rule{1ex}{0ex}}^{\small \mbox{prohibited}}$
}}
\thicklines
\qbezier(11,43)(20,36)(60,36.5)
\put(59,36.4){\vector(1,0){3}}
\color{red}
\put(15,43){\line(1,-1){14}}
\put(44,26){\line(1,1){16}}
\put(69,26){\line(1,1){16}}
\put(50,7){\makebox(0,0){{\small high}}}
\color{blue}
\put(20,43){\line(1,-1){14}}
\put(39,26){\line(1,1){16}}
\put(64,26){\line(1,1){16}}
\put(34,45){\makebox(0,0){{\small low}}}
\end{picture}
\\[-.3ex]
{\small (a) Contradicting the separability.}
& {\small \begin{tabular}{c}(b) Necessity to encircle the equilibrium $x_f$ \\
when including points arbitrarily \\
close to $x_1$-axis: a contradiction.\end{tabular}} 
\end{tabular}
\vspace{-1.2ex}
\end{center}
\caption{Obstacles in constructing a strict Lyapunov function in terms of level sets: 
The lines and the arrows are segments of level sets and trajectories, respectively.}
\label{fig:keys}
\end{figure}
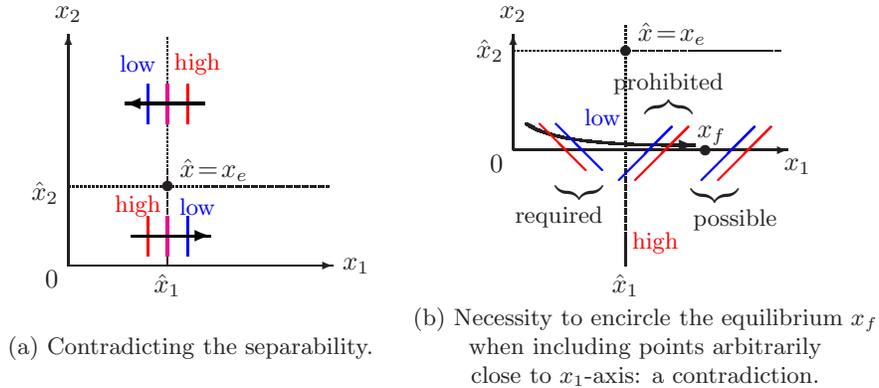
\section{Concluding Remarks}\label{sec:conc}

This paper has proved ISS of the SIR model with respect to perturbation of 
the newborn/immigration rate in both the endemic and the disease-free scenarios. 
The establishment is based on the construction of ISS Lyapunov functions. 
The functions play the role of traditional Lyapunov functions when 
the newborn/immigration rate is constant. It has been discussed that 
the proposed Lyapunov functions give the largest possible estimate of the domain of 
attraction and the ultimate boundedness in a qualitative sense. 
The developments do not rely on the simplifying assumptions which are 
often employed in the literature. The derivative of the proposed 
Lyapunov functions is strictly negative everywhere in 
sublevel sets of the Lyapunov functions except at the target equilibrium. 
This has allowed us to bypasses LaSalle's 
invariance principle, and to establish ISS addressing the perturbation. 
This paper has elaborated the construction of Lyapunov functions 
by distilling essential difficulties posed by the SIR model. 

It seems that no attention had been paid to ISS of the SIR model 
with respect to perturbation of the newborn/immigration rate, 
i.e., robustness of the endemic equilibrium and the disease-free equilibrium. 
Proving the ISS property had not been possible either since 
Lyapunov functions were not strict \cite{MalFredStrLyapBook09}, due to the reason clarified in 
Section \ref{sec:keys}. 
The robustness of the endemic equilibrium may sound undesirable 
in view of preventing disease spread. Nevertheless, controlling the 
peak and lowering the steady-state level of the infected population  
are beneficial to societies. The derivative of the ISS Lyapunov functions 
developed in this paper confirms that the increase of the death rate $\mu$ is 
the only almighty parameter that can not only reduce the peak and result in 
faster convergence, but also reduce the fluctuation of the state with respect 
to the perturbation of the newborn/immigration rate. 
It is also estimated that although the reduction of the transmission rate $\beta$ 
does not have such mighty effect.  it can simply 
avoid the endemic equilibrium or lower the steady-state 
level of the infected population. 
These are already known by using traditional local analysis and phase 
portraits. Nevertheless, the geometric structure revealed by 
the region partitioning and slopes of the proposed Lyapunov functions 
gives an insight into the flow of the populations 
in the SIR model globally in the state space. 
Importantly and interestingly, the ISS property proved in this paper 
has confirmed a linear transition of the magnitude of 
the state variables with respect to the perturbation magnitude 
of the newborn/immigration rate globally in spite of 
the bifurcation from the disease-free equilibrium to 
the endemic equilibrium and vise versa.

Needless to say, Lyapunov functions are known to be useful for designing 
controllers, and investigating control design for the SIR model 
is the most important direction of the future research. 
To this end, the proposed Lyapunov functions aiming at geometric 
understanding the SIR model can be modified into functions which ease the 
construction of controllers by smoothing out the edges of switching \cite{JIA96LYA}. 
In fact, the gradient-type design \cite{SONuniv98,FREEBK} based on a non-smooth Lyapunov function 
results in a discontinuous controller, and the notion of the system solution and 
the derivative need to be adjusted mathematically \cite{BACROSLiapbook05}. 
Bypassing such technicalities would be practically advantageous. 

%

\medskip
\medskip

\end{document}